\documentclass[13pt,oneside]{article}
\usepackage{times,amsmath,amsbsy,amssymb,amscd,mathrsfs,graphicx}
\usepackage{amssymb}
\usepackage{amsmath,bm}
\usepackage{amsmath,amsthm}

\usepackage{subcaption}
\usepackage{framed}
\usepackage{xcolor,color}
\usepackage[font=small,labelfont=md,textfont=it]{caption}
\usepackage{floatrow,float}
\usepackage[titletoc, title]{appendix}
\usepackage[colorlinks,linkcolor=blue,citecolor=blue]{hyperref}
\usepackage{longtable}
\usepackage{pgfplots}
\usepackage{diagbox}
\usepackage{booktabs,makecell,multirow}
\usepackage[capitalise, nosort]{cleveref}
\usepackage{algorithm}
\usepackage{algorithmic}
\usepackage{amsmath,bm}
\usepackage{cases}
\usepackage{geometry}
\usepackage{extarrows}
\usepackage{tcolorbox}
\usepackage{bbm}
\usepackage{syntonly}
\usepackage{siunitx}
\usepackage[figuresright]{rotating}
\usepackage{epstopdf}

\crefname{equation}{}{}
\crefname{lem}{Lemma}{Lemmas}
\crefname{coro}{Corollary}{Corollaries}
\crefname{thm}{Theorem}{Theorems}

\newcommand{\dd}{\,{\rm d}}

\newcommand{\R}{\,{\mathbb R}}
\newcommand{\diag}[1]{{\rm diag}\left({#1} \right)}

\newcommand{\dual}[1]{\left\langle {#1} \right\rangle}

\newcommand{\argmin}[0]{ {\mathop{{\rm  argmin}}\,}}

\newcommand{\argmax}[0]{ {\mathop{{\rm  argmax}}\,}}

\newcommand{\nm}[1]{\left\lVert {#1} \right\rVert}
\newcommand{\snm}[1]{\left\lvert {#1} \right\rvert}

\newcommand{\ssnm}[1]
{
	\left\vert\kern-0.25ex
	\left\vert\kern-0.25ex
	\left\vert
	{#1}
	\right\vert\kern-0.25ex
	\right\vert\kern-0.25ex
	\right\vert
}

\makeatletter
\def\spher@harm#1{%
	\vbox{\hbox{%
			\offinterlineskip
			\valign{&\hb@xt@2\p@{\hss$##$\hss}\vskip.2ex\cr#1\crcr}%
		}\vskip-.36ex}%
}
\def\gshone{\spher@harm{.}}
\def\gshtwo{\spher@harm{.&.}}
\def\gshthree{\spher@harm{.&.&.}}
\let\gsh\spher@harm
\makeatother

\newtheorem{coro}{Corollary}[section]

\newtheorem{prop}{Proposition}[section]

\newtheorem{lem}{Lemma}[section]

\newtheorem{rem}{Remark}[section]

\newtheorem{thm}{Theorem}[section]

\newtheorem{define}{Definition}

\newcolumntype{I}{!{\vrule width 1,5pt}}
\newlength\savedwidth

\newlength\savewidth

\newcounter{mnote}
\let\oldmarginpar\marginpar
\renewcommand\marginpar[1]
{\-\oldmarginpar[\raggedleft\footnotesize #1]{\raggedright\footnotesize #1}}

\makeatletter\def\@captype{table}\makeatother

\newfloatcommand{capbtabbox}{table}[][\FBwidth]
\usepackage[T1]{fontenc}
\usepackage[utf8]{inputenc}
\usepackage{authblk}
\usepackage[all]{xy}

\usepackage{graphicx}
\usepackage{tikz}
\tikzset{elegant/.style={smooth,thick,samples=50,black}}
\tikzset{eaxis/.style={->,>=stealth}}

\usepackage[misc]{ifsym}

\begin{document}
\title{
	\Large \bf A unified high-resolution ODE framework for first-order methods\thanks{This work was supported by the National Natural Science Foundation of China (Grant No. 12401402), the Science and Technology Research Program of Chongqing Municipal Education Commission (Grant Nos. KJZD-K202300505), the Natural Science Foundation of Chongqing (Grant No. CSTB2024NSCQ-MSX0329) and the Foundation of Chongqing Normal University (Grant No. 22xwB020).}}

\author[,1]{Lixia Wang \thanks{Email: 2023110510035@stu.cqnu.edu.cn}}
	\author[,1,2]{Hao Luo \thanks{Email: luohao@cqnu.edu.cn; luohao@cqbdri.pku.edu.cn}}
\affil[1]{National Center for Applied Mathematics in Chongqing, Chongqing Normal University, Chongqing, 401331, China}
\affil[2]{Chongqing Research Institute of Big Data, Peking University,  Chongqing, 401121, China}

	\date{\today}
	\maketitle
	
\begin{abstract}
For a generic discrete-time algorithm (DTA): $z^+=g(z,s)$, where $s$ is the step size, Lu (Math. Program., 194(1):1061--1112, 2022) proposed an $O(s^r)$-resolution ordinary differential equation (ODE) framework based on the backward error analysis, which can be used to analyze many DTAs satisfying the fixed point assumption $g(z,0)=z$ such as gradient descent, extra gradient method and primal-dual hybrid gradient (PDHG). However, most first-order methods with momentum violate this critical assumption. To address this issue, in this work, we introduce a novel $O((\sqrt{s})^r)$-resolution ODE framework for accelerated first-order methods allowing momentum and variable parameters, such as Nesterov accelerated gradient (NAG), heavy-ball (HB) method and accelerated mirror gradient. The proposed high-resolution framework provides deeper insight into the convergence properties of DTAs. Especially, although the $O(1)$-resolution ODEs for HB and NAG are identical, their $O(\sqrt{s})$-resolution ODEs differ from the subtle existence of the Hessian-driven damping. Moreover, we propose a high-resolution correction approach and apply it to PDHG and HB for provably convergent modifications that achieve global optimal convergence rates. Numerical results are reported to confirm the theoretical predictions.
\end{abstract}
	
	
\section{Introduction}
\label{sec:intro}
In recent years, first-order optimization methods have attracted much attention because of its wide applications to many research fields such as data science, machine learning and image processing. Generally speaking, first-order methods can be recast into an abstract one-step discrete-time algorithm (DTA)
\begin{equation}\label{eq:z+g}
	z^+=g(z,s),
\end{equation}
where $s>0$ denotes the step size and $g:\mathcal Z\times\R_+\to\mathcal Z$ stands for the iterative mapping. Here and throughout, $\mathcal Z$ is a finite dimensional Hilbert space with the inner product $\dual{\cdot,\cdot}$ and the induced norm $\nm{\cdot}=\sqrt{\dual{\cdot,\cdot}}$. 

Such a template \cref{eq:z+g} can be naturally recognized as  proper numerical discretization of some underlying continuous-time ordinary differential equation (ODE). In fact, there has been a historically recognized connection between DTAs and ODEs.
The simplest case is the gradient descent (GD) which aims to minimize a
function $F:\R^n\to\R$ via the update
\begin{equation}\label{eq:gd}
	\tag{GD}
	x_{k+1}=x_k-s\nabla F(x_k).
\end{equation}
The continuous limit leads to the gradient flow
\begin{equation}\label{eq:gd-gf}
	x'+\nabla F(x)=0.
\end{equation}
Under proper smooth assumptions, we have the one-step local error bound $\nm{x(s)-x_1}\leq O(s^2)$. In addition, it is known that the differential equation \cref{eq:gd-gf} converges with the decay rate $F(x(t))-F^*\leq O(1/t)$ for convex functions (cf.\cite{chen_luo_unified_2021}). Interestingly, \cref{eq:gd} admits a sublinear rate $F(x_k)-F^*\leq O(1/k)$, which is consistent with the continuous flow. In other words, the gradient flow \cref{eq:gd-gf} matches \cref{eq:gd} locally and globally. From this point of view, the continuous dynamical system approach not only gives an alternate way to understand DTAs but also provides more tools to design and analyze  discrete algorithms, especially first-order optimization methods.

\subsection{Low-resolution models}
\label{sec:intro-low}
Except for the simple case \cref{eq:gd}, there are two well-known first-order methods, the heavy ball (HB) method and Nesterov's accelerated gradient (NAG), attracting many attentions to the investigations on both discrete-time and continuous-time levels \cite{attouch_fast_2018,luo_differential_2022,siegel_accelerated_2019,wilson_lyapunov_2021,shi_understanding_2021,su_dierential_2016,wibisono_variational_2016}. In the following, we mainly focus on low-resolution models of these two methods. For more low-resolution works related to first-order primal-dual methods, we refer to 
\cite{Luo_acclerated_2024,luo_universal_2024,luo_unified_2025,luo_acc_primal-dual_2021,luo_icpdps_2025,Li2024,Li2024a,luo_primal-dual_2022,Apidopoulos2025}. 

In \cite{Polyak_Some_1964}, Polyak studied the HB method
\begin{equation}\label{eq:hb}\tag{HB}
	x_{k+1} = x_k + \beta(x_{k}-x_{k-1}) - s\nabla F(x_k).
\end{equation} 
For smooth and strongly convex objective $F\in \mathcal{S}^{2,1}_{\mu,L}$ (cf.\cref{sec:notation}), the local optimal linear rate $O((1-\sqrt{1/\kappa})^k)$ has been established in \cite[Theorem 9]{Polyak_Some_1964} via spectrum analysis, and the optimal parameters for quadratic objectives are 
\begin{equation}\label{eq:s-beta-hb}
	s_{\rm hb}=4(\sqrt{L}+\sqrt{\mu})^{-2}\quad\text{and}\quad\beta_{\rm hb}=(1-\sqrt{\mu s_{\rm hb}})^2.
\end{equation}
This setting results in a low-resolution ODE for \cref{eq:hb}:
\begin{equation}\label{eq:hb-o1}
	x''+2\sqrt{\mu }x'+\nabla F(x)=0,
\end{equation}
which converges exponentially $F(x(t))-F^*\leq O(e^{-2\sqrt{\mu}t})$ as long as $F\in\mathcal S_{\mu,L}^1$; see \cite{luo_differential_2022,siegel_accelerated_2019,wilson_lyapunov_2021}. For an alternate setting (almost identical to \cref{eq:s-beta-hb})
\begin{equation}\label{eq:s-beta-shi}
	s=1/L\quad\text{and}\quad \beta=(1-\sqrt{\mu s})/(1+\sqrt{\mu s}),
\end{equation}
Shi et al. \cite{shi_understanding_2021} derived the following model
\begin{equation}\label{eq:hb-os-shi}
	x''+2\sqrt{\mu }x'+\left(1+\sqrt{\mu s}\right)\nabla F(x)=0,
\end{equation}
and established the decay rate $F(x(t))-F^*\leq O(e^{-\sqrt{\mu}t/4})$ for $F\in\mathcal S_{\mu,L}^{2,1}$.

However, Lessard et al. \cite{lessard_analysis_2016} constructed a counterexample demonstrating that Polyak's optimal choice \cref{eq:s-beta-hb} for \cref{eq:hb} does not ensure global convergence for general strongly convex objectives. For proper range of parameters $(\beta, s)$, the global {\it suboptimal} linear rate $O((1-1/\kappa)^k)$ has been proved by \cite{ghadimi_global_2015, sun_non-ergodic_2019,shi_understanding_2021}. 
Later on, Goujaud et al. \cite{goujaud2025provable} showed that \cref{eq:hb} provably fails to achieve the optimal rate $O((1-\sqrt{1/\kappa})^k)$ for $F\in \mathcal{S}^1_{\mu,L}$. Recently, Wei and Chen \cite{wei2024accelerated} introduced the accelerated over-relaxation heavy-ball (AOR-HB) method that applies the over-relaxation technique to the gradient term in \cref{eq:hb}, and established a provably global accelerated linear rate $O((1-\sqrt{1/\kappa})^k)$. Nevertheless, we claim that the continuous-time heavy-ball models \cref{eq:hb-o1,eq:hb-os-shi} match \cref{eq:hb} locally but not globally, and there is still a gap between \cref{eq:hb} and its continuous level:  \\
\textit{{\bf Question 1}: why does the low-resolution models \cref{eq:hb-o1,eq:hb-os-shi} converge while the discrete case \cref{eq:hb} does not yield optimal rate or even diverges?}

The next major development was due to Nesterov, who discovered an accelerated gradient method  \cite{Nesterov1983AMF,Nesterov2004}:
\begin{equation}\label{eq:nag}\tag{NAG}
x_{k+1} = x_k + \beta_k(x_{k}-x_{k-1}) - s\nabla F(x_k + \beta_k(x_{k}-x_{k-1})), 
\end{equation}
which differs from \cref{eq:hb} in the gradient term.
With appropriate
choices of $\beta_{k}$, \cref{eq:nag} converges with the sublinear rate $O(1/k^2)$ for $F\in  \mathcal{F}_L^{1}$ and the linear rate $O((1-\sqrt{1/\kappa})^k)$ for $F\in  \mathcal{S}_{\mu,L}^{1}$,  achieving the optimal complexity of first-order methods. From a continuous-time perspective, Su et al. \cite{su_dierential_2016} derived the low-resolution model of \cref{eq:nag} with $\beta_{k}=k/(k+3)$ and $s=1/L$ (in this case, the method is abbreviated as NAG-C) 
\begin{equation}\label{eq:avd}
x''+\frac{3}{t}x'+\nabla F(x)=0,
\end{equation}
which yields the decay rate $F(x(t))-F^*\leq O(1/t^2)$ for $F\in\mathcal F^1_L$. This is also called the asymptotically vanishing damping model due to Attouch et al. \cite{attouch_rate_2019,attouch_fast_2018}. For strongly convex objectives, one can adopt the constant choice \cref{eq:s-beta-shi} (in this case, the method is abbreviated as NAG-SC), which is very close to the optimal choice \cref{eq:s-beta-hb} of \cref{eq:hb} and yields the identical low-resolution ODE \cref{eq:hb-o1}; see \cite{luo_differential_2022,siegel_accelerated_2019,wilson_lyapunov_2021,shi_understanding_2021} and \cref{rem:compare}. Therefore, we can not distinguish \cref{eq:nag,eq:hb} from the continuous model \cref{eq:hb-o1} and there comes the following question: \\
\textit{{\bf Question 2}: how can we find the difference between \cref{eq:nag,eq:hb} from the continuous level?}

\subsection{High-resolution models}
In addition to low-resolution models, the high-resolution approach has also been applied to DTAs, for better capturing the behaviors and properties of the discrete case.  An interesting work by Shi et al. \cite{shi_understanding_2021} showed that for NAG-SC,  a careful high-order Taylor expansion yields the high-resolution model
\begin{equation}\label{eq:os-nag-sc}
x''+2\sqrt{\mu }x'+(1+\sqrt{\mu s})\nabla F(x)+\sqrt{s}\nabla^2F(x)x'=0,
\end{equation}
which converges with the rate $F(x(t))-F^*\leq O(e^{-\sqrt{\mu}t/4})$ for $F\in\mathcal S_{\mu,L}^{2,1}$. Similarly, there is also a high-resolution model for NAG-C:
\begin{equation}\label{eq:os-ode-nag-c}
x''+\frac{3}{t}x'+\left(1+\frac{3\sqrt{s}}{2t}\right)\nabla F(x)+\sqrt{s}\nabla^2F(x)x'=0,
\end{equation}
with the rate $F(x(t))-F^*\leq O(1/t^2)$ for $F\in\mathcal F^2_L$. Compared with the low-resolution models \cref{eq:hb-o1,eq:avd} for NAG-SC and NAG-C, both \cref{eq:os-nag-sc,eq:os-ode-nag-c} have additional $O(\sqrt{s})$-terms involving the Hessian information. This is the so-called {\it gradient correction} \cite[Section 1.1]{shi_understanding_2021} explaining why \cref{eq:nag} is more stable than \cref{eq:hb}. We note that this is very close to the mechanism of the {\it Hessian-driven damping} by \cite{Attouch2020f,chen_first_2019}.  

Instead of directly  using the Taylor expansion, the {\it backward error analysis} \cite{Hairer_2006_Geometric} provides a more systematic approach for analyzing DTAs. It aims to find a modified of the low-resolution model that are more close to the given DTA. The implicit gradient regularization \cite{Barrett2021}  gives a high-resolution model of \cref{eq:gd}:
\begin{equation}\label{eq:os-gd}
x' + \nabla F(x)+\frac{s}{2}\nabla^2 F(x)\nabla F(x) = 0,
\end{equation}
with a high-order local error $\nm{x(s)-x_1}\leq O(s^3)$. Lu \cite{lu_osr-resolution_2022} proposed an $O(s^r)$-resolution ODE framework for the abstract DTA template \cref{eq:z+g} with the {\it fixed-point assumption}
\begin{equation}\label{eq:fp-assum}
g(z,0) = z\quad\forall\,z\in\mathcal Z.
\end{equation}
This covers many existing first-order methods for unconstrained problems and minimax problems such as proximal point algorithm, proximal gradient method, gradient descent-ascent (GDA), extra-gradient method (EGM) and primal-dual hybrid gradient (PDHG). As we all know, GDA is divergent even for convex-concave minimax problems while EGM converges. In \cite[Section 2.2]{lu_osr-resolution_2022}, Lu found that the $O(1)$-resolution ODEs of these two methods are the same but the $O(s)$-resolution models are different. This subtle difference provides better understanding on GDA and EGM and results in a new algorithm, called the Jacobian method, which is based on the high-order correction of GDA and applied to bilinear minimax problems.  

However, as noted at the end of \cite[Section 2.2]{lu_osr-resolution_2022}, the current $O(s^r)$-resolution ODE framework cannot be applied to first-order methods with momentum which violates the fixed-point assumption \cref{eq:fp-assum}:
\textit{\begin{quote}
	``However, this framework does not apply directly to Nesterov's accelerated method for minimizing a strongly-convex function, because  \text{$g(z, 0) \neq z$} due to the existence of the momentum term in the algorithm, which violates our assumption on the function \text{$g(z, s)$}.''
\end{quote}
} 
Therefore, it is of interest to extend the such a framework to accelerated first-order methods and here comes another question: \\
\textit{{\bf Question 3}: how to develop a high-resolution ODE framework for first-order methods with momentum and variable parameters?}

\subsection{Main contributions}

Focusing on the three questions mentioned in the last two sections, based on the $O(s^r)$-resolution idea from \cite{lu_osr-resolution_2022}, we proposed a unified $O((\sqrt{s})^r)$-resolution ODE framework for accelerated first-order methods. The key is to transform an accelerated gradient method into the DTA template \cref{eq:z+g} 
with the step size $\sqrt{s}$ (instead of $s$):
\[
X^+ = \Phi(X,\sqrt{s}),
\]
where the mapping $\Phi:\mathcal X\times\R_+\to\mathcal X$ satisfies $\Phi(X,0) = X$ for all $X\in\mathcal X$. To further explain the main idea, we introduce $v_k = (x_k-x_{k-1})/\sqrt{s}$ and rewrite  \cref{eq:hb} as follows
\begin{equation*}
\left\{
\begin{aligned}
	x_{k+1} = {}&x_k+\sqrt{s}v_{k+1},\\
	v_{k+1} ={}& v_k+(\beta-1) v_k -\sqrt{s}\nabla F(x_k).
\end{aligned}
\right. 
\end{equation*}
It is clear that $(x_{k+1},v_{k+1})=\Phi(x_k,v_k,\sqrt{s})$ with 
\[
\Phi(X,\sqrt{s}):=\begin{bmatrix}
x+\sqrt{s}\left[	\beta v-\sqrt{s}\nabla F(x)\right]\\
\beta v-\sqrt{s}\nabla F(x)
\end{bmatrix},\quad X= (x,v).
\]
Assume $\beta=\beta(s)$ is smooth such that $\lim_{s\to 0}\beta(s) =1$, then  $\Phi(X,0) = X$, thereby satisfying the fixed point assumption \cref{eq:fp-assum}.
This novel transformation technique overcomes the challenges posed by momentum and variable parameters in accelerated first-order methods and gives a positive answer to {\bf Question 3}. 

With the proposed $O((\sqrt{s})^r)$-resolution framework, we systematically derive $O(\sqrt{s})$-resolution ODEs for first-order algorithms including HB, NAG, and accelerated mirror descent (AMD), yielding new insights into the comparisons between HB and NAG. This also provides convincible explains for {\bf Question 1} and {\bf Question 2} and rebuilds the results by Shi et al. \cite{shi_understanding_2021}: the hidden gradient correction effect or Hessian-driven damping term makes NAG more stable than HB, which only involves the velocity correction; see Remarks \ref{rem:hb-os} and \ref{rem:compare}.
As by products, we use the high-resolution correction idea to propose two convergent modifications of PDHG and HB, and prove the global optimal convergence rates via the Lyapunov analysis. 

\subsection{Notations and organization}
\label{sec:notation}
For $k\in\mathbb N$ and $d\in\mathbb N$, denote by $C^k(\R^d)$ the set of all $k$-times continuous differentiable functions on $\R^d$, and the subclass
$\mathcal F^k(\R^d)\subset C^k(\R^d)$ contains all convex functions in $C^k(\R^d)$. For every $f\in\mathcal F^1(\R^d)$, denote by $f^*$ the conjugate function of $f$. If $f\in\mathcal F^{k}(\R^d)$ has $L$-Lipschitz continuous gradient: $\| \nabla f(x) - \nabla f(y) \| \leq L \| x - y \|$ for all $x, y \in \mathbb{R}^d$, then we say $f\in\mathcal F_{L}^{k,1}(\R^d)$. When $k=1$, for simplicity, we write $\mathcal F_L^{1,1}(\R^d) = \mathcal F_L^1(\R^d)$. According to \cite[Theorem 2.1.5]{Nesterov2004}, for $f\in \mathcal{F}_{L }^1(\R^d)$, we have the estimate
\begin{equation}\label{eq:low-bound}
\dual{\nabla f(x)-\nabla f(y),x-y} \geq\frac{1}{L}\|\nabla f(x)-\nabla f(y)\|^2\quad\forall\,x,\,y\in\R^d.
\end{equation}

Moreover,  the function class $\mathcal{F}^2_L(\mathbb{R}^n)$
is the subclass of $\mathcal{F}^1_L(\mathbb{R}^n)$ such that each $f$ has a Lipschitz-continuous Hessian. Let $ \mathcal S_{\mu}^k(\R^d)\subset C^k(\R^d)$ be the set of all strongly convex functions in $C^k$ with the common convexity parameter $\mu>0$, which means for any $f\in\mathcal S_\mu^k(\R^d)$ we have $f(y) \geq f(x) + \langle \nabla f(x), y - x \rangle + \mu/2 \| y - x \|^2$ for all $x, y \in \mathbb{R}^d$. Denote by $ \mathcal S_{\mu,L}^{k,1}(\R^d):=\mathcal S_\mu^k(\R^d)\cap\mathcal F_L^1(\R^d)$. 
A key inequality for our subsequent analysis is that, for $f \in \mathcal{S}_{\mu,L}^1(\R^d)$,
\[ 
f(x)-f(y) \leq \dual{\nabla f(x)-\nabla f(y),x-y}- \frac{1}{2L}\|\nabla f(x)-\nabla f(y)\|^2\quad\forall\,x,\,y\in\R^d.
\]
When no confusion arises, we omit the underlying space $\R^d$ of the function classes.

The remainder of this paper is organized as follows. In \cref{sec:frame}, we review the high-resolution ODE framework from \cite{lu_osr-resolution_2022} and derive the corresponding ODEs for some typical first-order methods without momentum. The framework is extended in \cref{sec:Os-am}  to accelerated first-order methods via a novel transformation technique, with examples illustrating its applicability. Subsequently, in \cref{sec:ode to dta-pdhg,sec:ode to dta-hb}, we propose proper modifications of HB and PDHG via time discretizations of the high-resolution corrected ODEs, and establish the corresponding convergence rates by tailored  Lyapunov functions. In \cref{sec:ex-hb}, some numerical tests are provided to validate our theoretical results. Finally, concluding remarks are summarized in \cref{sec:con}.

\section{The $O(s^r)$-Resolution Framework}
\label{sec:frame}
In this part, we revisit the high-resolution framework \cite{lu_osr-resolution_2022}. In \cref{sec:osr}, following \cite[Section 2]{lu_osr-resolution_2022}, we introduce the so-called $O(s^r)$-resolution ODE framework for a given DTA \cref{eq:z+g} with the fixed-point assumption $g(z,0) = z$. After that, we provide in \cref{sec:osr-gd} a detailed investigations on the $O(s^r)$-resolution ODEs of several DTAs.  
\subsection{The high-resolution ODE}
\label{sec:osr}
Firstly, let us recall the definition of the $O(s^r)$-resolution ODE from \cite[Section 2]{lu_osr-resolution_2022}.
\begin{define}[\cite{lu_osr-resolution_2022}]
	\label{def:osr}
	For a given DTA \cref{eq:z+g} with $g(z,0)=z$ for all $z\in\mathcal Z$, if there is an ODE system with the following format
	\begin{equation}\label{eq:osr}
		Z'(t) = f_0(Z(t)) + sf_1(Z(t))+\cdots+s^rf_r(Z(t)),\quad Z(0)=z,
	\end{equation}
	that satisfies $\nm{Z(s)-z^+} = o\left(s^{r+1}\right)$ with $r\geq 0$, then we call \cref{eq:osr} the $O(s^r)$-resolution ODE of the DTA \cref{eq:z+g}.
\end{define}
By \cite[Theorem 1]{lu_osr-resolution_2022}, the $O(s^r)$-resolution ODE of the DTA \cref{eq:z+g}  with $g(z,0)=z$ exists uniquely. This has been summarized in the following theorem.
\begin{thm}[\cite{lu_osr-resolution_2022}]
	\label{thm:unique}
	Given a DTA \cref{eq:z+g} with a sufficiently smooth  mapping $g:\mathcal Z\times \R_+\to\mathcal Z$ such that $g(z,0)=z$ for all $z\in\mathcal Z$, then its $O(s^{r})$-resolution ODE \cref{eq:osr} exists uniquely and is given by 
	\begin{equation}\label{eq:fj}
		f_{0}(z)={}g_1(z),\quad 
		f_{j}(z) ={}\frac{g_{j+1}(z)}{(j+1)!}- \sum_{k=2}^{j+1}\frac{1}{k!}h_{k,j+1-k}(z) ,\quad 1\leq j\leq r,
	\end{equation}
	where $g_{j}(z)=\partial^{j}_sg(z,s)\Big|_{s=0}$ for $0\leq j\leq r+1$ and $h_{j,i}:\mathcal Z\to\mathcal Z$ is defined recursively by 
	\begin{equation}
		\label{eq:hij}
		\begin{aligned}
			h_{1,i}(z) = {}&f_i(z),\quad
			h_{k,i}(z)={}\sum_{l=0}^{i}\nabla h_{k-1,l}(z)f_{i-l}(z),\quad 2\leq k\leq r+1,\quad 0\leq i\leq r.
		\end{aligned}
	\end{equation}
\end{thm} 

In fact, the $O(s^r)$-resolution ODE corresponds to the modified equation in the backward error analysis \cite{Hairer_2006_Geometric}, which provides a better continuous approximation to the DTA \cref{eq:z+g} of a local order $o(s^{r+1})$. In particular, if $g(z,s)$ is sufficiently smooth, then we have the following error estimates.
\begin{prop}\label{prop:0s-bound}
	Suppose $g$ is sufficiently smooth satisfying $g(z,0)=z$ for all $z\in\mathcal Z$. Let $\{z_k\}_{k=0}^N$ be generated by the DTA \cref{eq:z+g} with $z_0\in\mathcal Z$ and the $O(s^{r})$-resolution ODE be given by \cref{eq:osr}. If $s\leq s_0$ for some $s_0>0$, then there exist $\alpha_0:=(e^{s_0}-1)/s_0>0$ and $ C_0,C_1>0$ such that 
	\begin{itemize}
		\item  the local error bound $\nm{Z(s)-z_1} \leq C_1s^{r+2}+\left(1+\alpha_0C_0s\right)\nm{Z(0)-z_0}$, 
		\item the intermediate error bound $	\nm{Z(t_k)-z_k}
		\leq{}\frac{C_1}{\alpha_0C_0}e^{\alpha_0C_0T}s^{r+1}+e^{\alpha_0C_0T}\nm{Z(0)-z_0}$ for all $1\leq k\leq N$, where $T = Ns$.
	\end{itemize}
\end{prop}
\begin{proof} 
	Note that the right-hand side of \cref{eq:osr} determined by \cref{eq:fj}  is smooth enough. In particular, there exits some $C_0>0$ such that
	\begin{equation}\label{eq:Lip-fi}
		\nm{\sum_{i=0}^{r}s^if_i(\widetilde{Z}) -\sum_{i=0}^{r}s^if_i(Z)   }\leq C_0\|\widetilde{Z}-Z\|,
	\end{equation}
	for all $Z,\widetilde{Z}\in\mathcal Z$. Consider the following perturbed ODE 
	\[
	\widetilde{Z}'(t) = f_0(\widetilde{Z}(t)) + sf_1(\widetilde{Z}(t))+\cdots+s^rf_r(\widetilde{Z}(t)),\quad \widetilde{Z}(0)=z_0.
	\]
	Then by \cite[Remark 1]{lu_osr-resolution_2022}, we have $\|\widetilde{Z}(s)-z_1\| \leq C_1s^{r+2}$ for some $C_1>0$ independent of $s$. On the other hand, according to \cref{eq:Lip-fi}, it follows that
	\[
	\begin{aligned}
		\|\widetilde{Z}(t)-Z(t)\|={}&\nm{\widetilde{Z}(0)-Z(0)+\int_{0}^{t}\sum_{i=0}^{r}s^i\big[f_i(\widetilde{Z}(\tau)) -f_i(Z(\tau)) \big]\dd \tau}\\
		\leq{}&\|\widetilde{Z}(0)-Z(0)\|+C_0\int_{0}^{t}\|\widetilde{Z}(\tau)-Z(\tau)\|\dd \tau,
	\end{aligned}
	\]
	which, together with the Gronwall inequality, gives
	\[
	\|\widetilde{Z}(t)-Z(t)\|\leq \left(1+C_0(e^t-1)\right)\nm{z_0-Z(0)},\quad t>0.
	\]
	Consequently, we get the local error bound  immediately
	\[
	\begin{aligned}
		\nm{Z(s)-z_1}\leq \|Z(s)-\widetilde{Z}(s)\|+\|\widetilde{Z}(s)-z_1\|\leq{}& C_1s^{r+2}+\left(1+C_0(e^s-1)\right)\nm{Z(0)-z_0}\\
		\leq	{}& C_1s^{r+2}+\left(1+\alpha_0C_0s\right)\nm{Z(0)-z_0},
	\end{aligned}
	\]
	in view of the trivial inequality $e^s-1\leq \alpha_0 s$ for all $0<s\leq s_0$. A similar argument implies
	\[
	\nm{Z(t_k)-z_k}\leq  C_1s^{r+2}+\left(1+\alpha_0C_0s\right)\nm{Z(t_{k-1})-z_{k-1}},
	\]
	where $t_k = ks$ for all $1\leq k\leq N$. This also yields that
	\[
	\begin{aligned}
		\nm{Z(t_k)-z_k}\leq{}& C_1s^{r+1}\frac{\left(1+\alpha_0C_0s\right)^k-1}{\alpha_0C_0}+\left(1+\alpha_0C_0s\right)^k\nm{Z(0)-z_0}\\
		\leq{}&C_1s^{r+1}\frac{e^{\alpha_0C_0ks}}{\alpha_0C_0}+e^{\alpha_0C_0ks}\nm{Z(0)-z_0}
		\leq{}C_1s^{r+1}\frac{e^{\alpha_0C_0T}}{\alpha_0C_0}+e^{\alpha_0C_0T}\nm{Z(0)-z_0}.
	\end{aligned}
	\]
	This completes the proof.
\end{proof} 

Let us give more explanation about the calculation of the coefficient $f_j(z)$ for $1\leq j\leq r$. The recurrence relation \cref{eq:hij} can be written equivalently as the following matrix form
\[
\begin{aligned}
	\begin{bmatrix}
		h_{1,0}&	h_{1,1}&\cdots&	h_{1,r}\\
		h_{2,0}&h_{2,1}&\cdots&h_{2,r}\\
		\vdots&\vdots&\ddots&\vdots\\
		h_{r+1,0}&h_{r+1,1}&\cdots&h_{r+1,r}
	\end{bmatrix}
	={}&\begin{bmatrix}
		1&0&\cdots&0\\
		\nabla h_{1,0}&\nabla h_{1,1}&\cdots&\nabla h_{1,r}\\
		\nabla h_{2,0}&\nabla h_{2,1}&\cdots&\nabla h_{2,r}\\
		\vdots&\vdots&\ddots&\vdots\\
		\nabla h_{r,0}&\nabla h_{r,1}&\cdots&\nabla h_{r,r}
	\end{bmatrix} \begin{bmatrix}
		f_{0}&f_{1}&\cdots&f_{r}\\
		0&f_{0}&\cdots&f_{r-1}\\
		\vdots&\vdots&\ddots&\vdots\\
		0&0&\cdots&f_{0}
	\end{bmatrix}.
\end{aligned}
\]
Therefore, the computational flow for $h_{ij}$ is in column-wise: $f_j = h_{1,j}\to h_{2,j}\to\cdots\to h_{r+1,j}$ for $1\leq j\leq r$.
\subsection{Application to DTAs without momentum}
\label{sec:osr-gd}
In \cite[Section 2.1]{lu_osr-resolution_2022}, Lu has considered the application of the $O(s^r)$-resolution framework to three DTAs: 
\begin{itemize}
	\item gradient descent ascent (GDA): $z_{k+1} = z_k-sM(z_k)$,
	\item  proximal
	point method (PPM): $z_{k+1}=z_k-sM(z_{k+1})$,
	\item extra-gradient method (EGM): $z_{k+1} = z_k-sM(z_k-sM(z_k))$,
\end{itemize}
for solving the nonlinear
minimax problem $\min_{x\in\R^n}\max_{y\in\R^m}\,L(x,y)$, where $M(z):=[\nabla_xL(x,y),-\nabla_yL(x,y)]$ for all $z=(x,y)\in\R^n\times\R^m$. More precisely, the $O(s)$-resolution ODE of GDA is
\begin{equation}\label{eq:osr-gda}
	Z'=-M(Z)-\frac{s}{2}\nabla M(Z)M(Z),
\end{equation}
and the $O(s)$-resolution ODEs of PPM and EGM are the same one:
\begin{equation}\label{eq:osr-ppm}
	Z'=-M(Z)+\frac{s}{2}\nabla M(Z)M(Z).
\end{equation}

As mentioned at the end of \cite[Section 2.1]{lu_osr-resolution_2022}, it can be applied directly to many other first-order methods without momentum. For completeness and later use (cf.\cref{sec:ode to dta-pdhg}),  we provide a detailed investigations on the $O(s^r)$-resolution ODEs of more typical examples.
\subsubsection{Mirror descent}
\label{sec:osr-md}
As a generalization of the gradient descent to the non-Euclidean setting, the mirror descent (MD) reads as follows 
\begin{equation}\label{eq:md-xk}
	x_{k+1} = \mathop{\argmin}_{x\in\R^n} \left\{ s \langle \nabla F(x_k), x - x_k \rangle + D_\varphi(x, x_k) \right\},
\end{equation}
where $F\in\mathcal F^1(\R^n),\,s>0$ denotes the step size and  $D_\varphi(x,y):=\varphi(x)-\varphi(y)-\dual{\varphi(y),x-y}$ represents the Bregman divergence with respect to a given prox-function $\varphi\in \mathcal S_1^1(\R^n)$. 

Introduce the dual variable $z_k = \nabla\varphi(x_k)$ and rewrite  \cref{eq:md-xk} as a dual formulation
\begin{equation}\label{eq:md-zk}
	z_{k+1} = z_k -  s \nabla F(\nabla \varphi^*(z_{k})),
\end{equation}
which is a 
standard DTA $z^+=g(z,s)=z-s\nabla F(\nabla \varphi^*(z))$. Notice that $g(z,0) = z$ and
\[
g_0(z)=g(z,0)=z,\,g_1(z) =\partial_sg(z,0)= -\nabla F(\nabla\varphi^*(z)),\,g_j(z)=\partial^j_sg(z,0)=0,\,j\geq 2.
\]
According to \cref{thm:unique}, it follows that $	f_0(z) = g_1(z) =  -\nabla F(\nabla\varphi^*(z))$ and 
\[
\begin{aligned}
	f_1(z) = {}&g_2(z)-\frac{1}{2}h_{2,0}(z)=-\frac{1}{2}\nabla f_{0}(z)f_{0}(z)\\
	={}&-\frac{1}{2} \nabla^2 F(\nabla \varphi^*(z))\nabla^2\varphi^*(z)\nabla F(\nabla \varphi^*(z)).
\end{aligned}
\]
This leads to the $O(s)$-resolution ODE of the dual MD \cref{eq:md-zk}.
\begin{thm}\label{thm:osr-md-zk}
	Assume that $\varphi^*\in C^2(\R^n)$ and $F\in \mathcal F^2(\R^n)$, then the $O(s)$-resolution ODE of the dual MD \cref{eq:md-zk} is given by
	\[
	Z'=-\nabla F(\nabla \varphi^*(Z))-\frac{s}{2} \nabla^2 F(\nabla \varphi^*(Z))\nabla^2\varphi^*(Z)\nabla F(\nabla \varphi^*(Z)),
	\]
	with $Z(0)=z_0=\nabla \varphi(x_0)$.
\end{thm} 

\subsubsection{Primal-dual hybrid gradient}
We now focus on the primal-dual method
\begin{equation}\label{eq:pdhg}
	\left\{
	\begin{aligned}
		&x_{k+1} =\mathop{\argmin}_{x\in\R^n}\left\{\mathcal L(x,y_k)+\frac{1}{2s}\nm{x-x_k}^2\right\}, \\
		&y_{k+1} =\mathop{\argmax}_{y\in\R^m}\left\{\mathcal L(x_{k+1}+\theta(x_{k+1}-x_k),y)-\frac{1}{2s}\nm{y-y_k}^2\right\},
	\end{aligned}
	\right.
\end{equation}
for solving the bilinear saddle-point problem
\begin{equation}\label{eq:bisp}
	\min_{x\in \R^n} \max_{y\in \mathbb{R}^m }\,\mathcal L(x,y) := F(x) + \dual{y,Ax}-G(y),
\end{equation}
where $F\in\mathcal F^1(\R^n)$ and $G\in\mathcal F^1(\R^m)$ and $A:\mathbb{R}^n \to \mathbb{R}^m$  is a linear operator. The step size is $s>0$ and the extrapolation parameter $\theta\in[0,1]$. The case $\theta=0$ corresponds to the primal-dual hybrid gradient (PDHG) method by Esser et al. \cite{esser_general_2010} and the other case $\theta=1$ is the method of Chambolle and Pock (CP) \cite{chambolle_first-order_2011}.

For simplicity, introduce a monotone operator $M:\mathcal{Z} \to  \mathcal{Z} $ by that
\[
M(z) := \begin{pmatrix} 
	\nabla F(x) + A^\top y \\ 
	\nabla G(y) - Ax 
\end{pmatrix}, \quad \forall\, z=(x,y) \in \mathcal{Z} := \R^n \times \R^m .
\]
Then we obtain a more compact preconditioned PPA presentation of PDHG \cref{eq:pdhg}:
\begin{equation}
	\label{eq:ppa-pdhg}
	z_{k+1}=z_k-s(I+sQI_\theta)^{-1}M(z_{k+1}),\quad  
		Q:= \begin{bmatrix}
			O&A^\top \\-A&O
		\end{bmatrix},
	\end{equation}
	where $I_\theta=\diag{\theta I,-I}$. For simplicity, let $Q_\theta:=QI_\theta$. This leads to $z^+=g(z,s)=\left[I+s\left(I+sQ_\theta\right)^{-1}M\right]^{-1}(z)$ satisfying $g(z,0) = z$ for all $z\in\mathcal Z$. A useful expansion lemma is given below, which implies immediately the $O(s)$-resolution ODE of \cref{eq:pdhg}.
	\begin{lem}
		\label{lem:expan-pdhg}
		We have
		\begin{equation}\label{eq:expan-pdhg}
			\left[I+s\left(I+sQ_\theta\right)^{-1}M\right]^{-1}(z)=z-sM(z)+s^2	\left[\nabla M(z)+Q_\theta\right]M(z)+o(s^2).
		\end{equation}
	\end{lem}
	\begin{proof}
		Notice that
		\[
		\left(I+sQ_\theta\right)^{-1}(z) = z-sQ_\theta z+s^2Q_\theta^2z+o(s^{2}).
		\]
		Assume $	\left[I+s\left(I+sQ_\theta\right)^{-1}M\right]^{-1}(z)=\phi_0(z)+s\phi_1(z)+s^2\phi_2(z)+o(s^2)$, which gives
		\[
		\begin{aligned}
			z = {}&\left[I+s\left(I+sQ_\theta\right)^{-1}M\right]\left(\phi_0(z)+s\phi_1(z)+s^2\phi_2(z)+o(s^2) \right)\\
			={}&\phi_0(z)+s\phi_1(z)+s^2\phi_2(z)+o(s^2)\\
			{}&\quad +s\left[I-sQ_\theta+o(s)\right]M\left(\phi_0(z)+s\phi_1(z)+s^2\phi_2(z)+o(s^2) \right)\\
			={}&\phi_0(z)+s\phi_1(z)+s^2\phi_2(z)+o(s^2)\\
			{}&\quad +s\left[M(\phi_0(z))+s\nabla M(\phi_0(z))\phi_1(z)-sQ_\theta M(\phi_0(z))+o(s)\right] \\		
			={}&\phi_0(z)+s\left[\phi_1(z)+M(\phi_0(z))\right]+o(s^2)\\
			{}&\quad+s^2\left[\phi_2(z)+\nabla M(\phi_0(z))\phi_1(z)-Q_\theta M(\phi_0(z))\right].
		\end{aligned}
		\]
		This leads to $\phi_0(z)=z,\,\phi_1(z)=-M(\phi_0(z))=-M(z)$ and
		\[
		\phi_2(z)=-\nabla M(\phi_0(z))\phi_1(z)+Q_\theta  M(\phi_0(z))=
		\nabla M(z)M(z)+Q_\theta M(z),
		\]
		which verifies \cref{eq:expan-pdhg} and completes the proof.
	\end{proof}
	\begin{thm}\label{thm:pdhg_cp_ode}
		Assume $F\in\mathcal F^2(\R^n)$ and $G\in\mathcal F^2(\R^m)$. The $O(s)$-resolution of PDHG \cref{eq:pdhg} is given by
		\begin{equation}\label{eq:osr-pdhg}
			Z'=-M(Z)+\frac{s}{2}\left[ \nabla M(Z)+2Q_\theta\right]M(Z),
		\end{equation}
		with $Z(0)=z_0=(x_0,y_0)\in\mathcal Z$.
	\end{thm}
	\begin{proof}
		Thanks to \cref{lem:expan-pdhg}, we have  
		\[
		g_0(z) = z,\,g_1(z) = -M(z),\,g_2(z)=	\left[\nabla M(z)+Q_\theta\right]M(z).
		\]
		Then invoking \cref{thm:unique}, it follows that $	f_0(z) = g_1(z) =  -M(z)$ and 
		\[
		\begin{aligned}
			f_1(z) = {}&g_2(z)-\frac{1}{2}h_{2,0}(z)=g_2(z)-\frac{1}{2}\nabla f_{0}(z)f_{0}(z)=\frac{1}{2}\left[ \nabla M(z)+2Q_\theta\right]M(z).
		\end{aligned}
		\]
		This gives \cref{eq:osr-pdhg} and finishes the proof. 
	\end{proof}
	
	From Theorem \ref{thm:pdhg_cp_ode}, the $O(1)$-resolution ODEs of PDHG $(\theta=0)$ and CP $(\theta=1)$ are the same one 
	\begin{equation}\label{eq:pdhg-o1}
		Z' = -M(Z).
	\end{equation}
	However, the convergence behaviors of these two methods are totally different; see \cref{fig:pdhg_cp} for a simple two dimensional illustration. The $O(1)$-resolution exhibits a limit cycle, which coincides with the non-convergence of PDHG (cf.\cite{he_convergence_2014}) but violates the convergence of CP. 
	\begin{figure}[H]
		\centering  
		{
			\begin{minipage}[t]{0.35\textwidth}
				\centering          
				\includegraphics[width=0.9\textwidth]{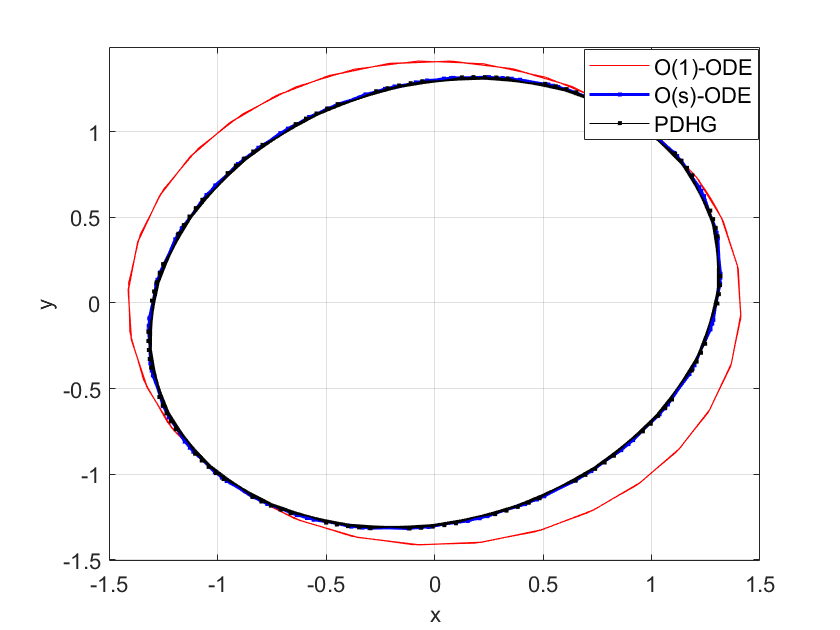}   
			\end{minipage}%
			\label{fig:pdhg_ode}
		}\centering  
		{
			\begin{minipage}[t]{0.35\textwidth}
				\centering          
				\includegraphics[width=0.9\textwidth]{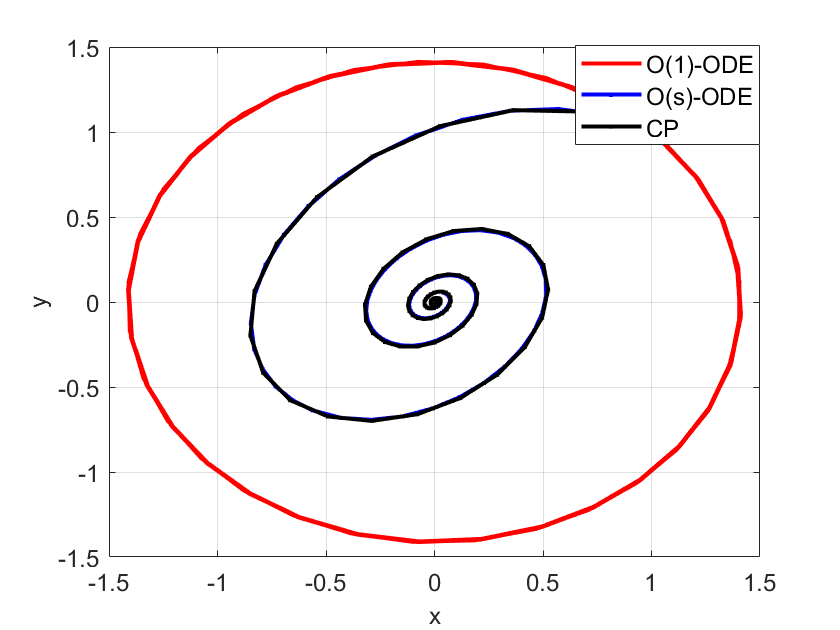}   
			\end{minipage}%
			\label{fig:cp_ode} 
		} 
		\caption{Illustration of PDHG and CP and  their corresponding resolution ODEs with the step size $s=0.3$. The saddle point function is  $L(x,y)=xy$ for $x,\,y\in \R$  and the initial state is $(x_0,y_0)=(1, 1)$.}
		\label{fig:pdhg_cp}
	\end{figure}
	\begin{table}[H]
		\centering
		\setlength{\tabcolsep}{4.5pt}
		\renewcommand{\arraystretch}{0.8}	
		\caption{Convergence rates of the $O(1)$-resolution ODEs of PDHG and CP }
		\label{tab:pdhg_comparison}
		\begin{tabular}{c c S[table-format=1.2e-2] S[table-format=1.2] S[table-format=1.2e-2] S[table-format=1.2] S[table-format=1.2e-2] S[table-format=1.2]}
			\toprule
			{Algor. / $O(1)$} & {$s$} & {$E_1(s)$} & {Rate} & {$E_2(s)$} & {Rate} & {$E_3(s)$} & {Rate} \\
			\midrule
			\multirow{5}{*}{PDHG / \text{\cref{eq:pdhg-o1}}}
			& $1/2^{5}$ & 9.39e-4 & {--} & 5.16e-1 & {--} & 6.72e-3 & {--} \\
			& $1/2^{6}$ & 2.66e-4 & 1.82 & 5.15e-1 & 0.01 & 3.35e-3 & 1.01 \\
			& $1/2^{7}$ & 7.12e-5 & 1.90 & 5.14e-1 & 0.00 & 1.67e-3 & 1.00 \\
			& $1/2^{8}$ & 1.84e-5 & 1.95 & 5.13e-1 & 0.00 & 8.33e-4 & 1.00 \\
			\midrule
			\multirow{5}{*}{CP / \text{\cref{eq:pdhg-o1}}}
			& $1/2^{5}$ & 1.22e-3 & {--} & 1.69e-1 & {--} & 3.59e-3 & {--}\\
			& $1/2^{6}$ & 3.44e-4 & 1.83 & 1.74e-1 & -0.04 & 1.84e-3 & 0.96 \\
			& $1/2^{7}$ & 9.16e-5 & 1.91 & 1.77e-1 & -0.02 & 9.35e-4 & 0.98 \\
			& $1/2^{8}$ & 2.36e-5 & 1.95 & 1.78e-1 & -0.01 & 4.71e-4 & 0.99 \\
			\bottomrule
		\end{tabular}
	\end{table}
	\begin{table}[H]
		\centering
		\setlength{\tabcolsep}{4.5pt}
		\renewcommand{\arraystretch}{0.85}	
		\caption{Convergence rates of the $O(s)$-resolution ODEs of PDHG and CP }
		\label{tab:pdhg_osarison_poly}
		\begin{tabular}{c c S[table-format=1.2e-2] S[table-format=1.2] S[table-format=1.2e-2] S[table-format=1.2] S[table-format=1.2e-2] S[table-format=1.2]}
			\toprule
			{Algor. / $O(s)$} & {$s$} & {$E_1(s)$} & {Rate} & {$E_2(s)$} & {Rate} & {$E_3(s)$} & {Rate} \\
			\midrule
			\multirow{5}{*}{PDHG / \text{\cref{eq:osr-pdhg}}}
			& $1/2^{5}$ & 8.73e-5 & {--} & 1.73e-2 & {--} & 2.58e-4 & {--} \\
			& $1/2^{6}$ & 1.26e-5 & 2.79 & 8.50e-3 & 1.02 & 6.43e-5 & 2.00 \\
			& $1/2^{7}$ & 1.70e-6 & 2.89 & 4.22e-3 & 1.01 & 1.61e-5 & 2.00 \\
			& $1/2^{8}$ & 2.21e-7 & 2.94 & 2.10e-3 & 1.01 & 4.02e-6 & 2.00 \\
			\midrule
			\multirow{5}{*}{CP / \text{\cref{eq:osr-pdhg}}}
			& $1/2^{5}$ & 1.23e-4 & {--} & 3.12e-2 & {--} & 5.22e-4 & {--} \\
			& $1/2^{6}$ & 1.76e-5 & 2.81 & 1.55e-2 & 1.01 & 1.30e-4 & 2.00 \\
			& $1/2^{7}$ & 2.35e-6 & 2.90 & 7.70e-3 & 1.01 & 3.25e-5 & 2.00 \\
			& $1/2^{8}$ & 3.05e-7 & 2.95 & 3.84e-3 & 1.00 & 8.14e-6 & 2.00 \\
			\bottomrule
		\end{tabular}
	\end{table}
	
	On the other hand, by \cref{eq:osr-pdhg}, the $O(s)$-resolution ODEs differ from the linear operator $Q_\theta$. For PDHG, it is asymmetric while for CP it is symmetric. This subtle difference leads to dramatically distinct behavior; see \cref{fig:pdhg_cp}. The $O(s)$-resolution ODE of PDHG approximates the discrete trajectory very well and performs still as a circle. Compared with this, the $O(s)$-resolution ODE of CP is also more close to its discrete trajectory and converges to the saddle point.

	Furthermore, for another simple example $L(x,y)=x^4+xy-y^2$, we verify the order of the convergence rate of the $O(s^{r})$-resolution ODEs ($r=0,1$) regarding to three measurements:
	\[
	E_1(s):=\nm{Z(s)-z_1},\quad E_2(s):=\sum_{k=1}^{N}\nm{Z(t_k)-z_k},\quad 
	E_3(s):=\sqrt{\sum_{k=1}^{N}s\nm{Z(t_k)-z_k}^2},
	\]
	where $t_k = ks$ for $1\leq k\leq N=T/s$ with $T=20$. According to \cref{prop:0s-bound}, when the initial values of the DTA and the corresponding $O(s)$-resolution ODE are identical, then we have $E_1(s)=O(s^{r+2}),E_2(s)=O(s^{r})$ and $E_3(s)=O(s^{r+1})$, which are verified by the numerical results in \cref{tab:pdhg_comparison,tab:pdhg_osarison_poly}.  

\section{The $O(\left(\sqrt{s}\right)^{r})$-Resolution Framework for Accelerated Methods}\label{sec:Os-am}
For general cases without the fixed-point assumption \cref{eq:fp-assum}, which correspond to acceleration methods with momentum, we define the $O((\sqrt{s})^r)$-resolution ODE by using proper equivalent template with step size $\sqrt{s}$. In this section, we derive high-resolution ODEs for a class of accelerated first-order methods. By leveraging the transformation technique, we extend the \(O(s^r)\)-resolution framework to accelerated methods with momentum and variable parameters, including HB, NAG and AMD. 
\subsection{The high-resolution ODE}
Let us extend \cref{def:osr} to DTAs without the fixed-point assumption $g(z,0)=z$.
\begin{define}	
	\label{def:osr-agd}
	For a given DTA \cref{eq:z+g}, if there exists an equivalent template
	\begin{equation}\label{eq:X+Gamma}
		X^+ = \Phi(X,\sqrt{s}),
	\end{equation}
	where $\Phi:\mathcal X\times\R_+\to\mathcal X$ satisfies $\Phi(X,0)=X$ for all $X\in\mathcal X$, and an ODE system with the following format
	\begin{equation}\label{eq:osr-agd}
		X'(t) = \Gamma_0(X(t)) + \sqrt{s}\Gamma_1(X(t))+\cdots+(\sqrt{s})^r\Gamma_r(X(t)),\quad 
		X(0)=X,
	\end{equation}
	that satisfies $\nm{X(\sqrt{s})-X^+} = o\left((\sqrt{s})^{r+1}\right)$ with $r\geq 0$, then we call \cref{eq:osr-agd} the $O((\sqrt{s})^r)$-resolution ODE of the DTA \cref{eq:z+g} with respect to the equivalent template \cref{eq:osr-agd}.
\end{define}
Thanks to \cref{thm:unique}, the $O((\sqrt{s})^r)$-resolution ODE of the DTA \cref{eq:z+g} with respect to the equivalent template \cref{eq:osr-agd} exists uniquely. 
\begin{thm} 
	\label{thm:unique-osr-agd}
	Given a DTA \cref{eq:z+g}, if there exists an equivalent template \cref{eq:X+Gamma} with  a  sufficiently smooth mapping $\Phi:\mathcal X\times\R_+\to\mathcal X$ satisfying $\Phi(X,0)=X$, then the $O((\sqrt{s})^r)$-resolution ODE of the DTA \cref{eq:z+g} with respect to the equivalent template \cref{eq:X+Gamma} exists uniquely and is given by 
	\begin{equation*}
		\begin{aligned}
			\Gamma_{0}(X)={}&\Phi_1(X),\quad 
			\Gamma_{j}(X) ={}\frac{\Phi_{j+1}(X)}{(j+1)!}- \sum_{k=2}^{j+1}\frac{1}{k!}H_{k,j+1-k}(X) ,\quad 1\leq j\leq r,
		\end{aligned}
	\end{equation*}
	where $\Phi_{j}(X)=\partial^{j}_\tau\Phi(X,\tau)\Big|_{\tau=0}$ for $0\leq j\leq r+1$ and $H_{j,i}:\mathcal X\to\mathcal X$ is defined recursively by 
	\[
	H_{1,i}(X) = {}\Gamma_i(X),\,
	H_{k,i}(X)={}\sum_{l=0}^{i}\nabla H_{k-1,l}(X)\Gamma_{i-l}(X),\quad 2\leq k\leq r+1,\,0\leq i\leq r.
	\]
\end{thm}  
\begin{rem}\label{rem:0ss-bound}
	Thanks to \cref{prop:0s-bound}, for the sequence $\{X_k\}_{k=0}^N$ generated by \cref{eq:X+Gamma}, we have the following  error estimates
	\begin{itemize}
		\item $\nm{X(\sqrt{s})-X_1}\leq C_1s^{\frac{r+2}{2}}+C_2\nm{X(0)-X_0}$,
		\item $\nm{X(t_k)-X_k}\leq C_3s^{\frac{r+1}{2}}+C_4\nm{X(0)-X_0}$ for all $1\leq k\leq N$, where $T = N\sqrt{s}$.
	\end{itemize}
\end{rem}
\subsection{Analysis of the heavy-ball method}
\label{sec:osr-hb}
We firstly focus on the heavy-ball method (cf.\cref{eq:hb})
\begin{equation}\label{eq:hb-osr}
	x_{k+1} = x_k + \beta_{\rm hb}(x_{k}-x_{k-1}) - s\nabla F(x_k).
\end{equation}
\begin{lem}\label{lem:hb}
	The HB iteration \cref{eq:hb-osr} is equivalent to  
	\begin{equation}\label{eq:hbv-x}
		X_{k+1} = \Phi_{\rm HB}(X_k,\sqrt{s}),
	\end{equation}
	where $X_k=(x_k,v_k)$ and the mapping $\Phi:\R^n\times\R^n\times\R_+\to\R^n\times\R^n$ is defined by
	\[
	\Phi_{\rm HB}(X,\sqrt{s}):=\begin{bmatrix}
		x+\sqrt{s}\left[	\beta_{\rm hb} v-\sqrt{s}\nabla F(x)\right]\\
		\beta_{\rm hb} v-\sqrt{s}\nabla F(x)
	\end{bmatrix},\quad\forall\, X= (x,v)\in\R^n\times\R^n.
	\]
	Moreover, if $\beta_{\rm hb} = \beta_{\rm hb}(\sqrt{s})$ is a smooth function such that $\lim_{s\to 0}\beta_{\rm hb}(\sqrt{s}) =1$, then $\Phi_{\rm HB}(X,0) = X$ for all $X\in\R^n\times\R^n$.
\end{lem}
\begin{proof}
	Introduce an auxiliary variable $	v_k =(x_k - x_{k-1})/\sqrt{s}$.
	This implies immediately the relation
	$x_{k+1} = x_k + \sqrt{s} v_{k+1}$. Combining this with the HB iteration \cref{eq:hb-osr} gives
	\[
	v_{k+1} = \frac{x_{k+1} - x_k}{\sqrt{s}} =  \beta_{\rm hb} v_k - \sqrt{s} \nabla F(x_k).
	\]
	Consequently, we obtain
	\[
	\left\{
	\begin{aligned}
		v_{k+1} ={}&  \beta_{\rm hb}  v_k -\sqrt{s}\nabla F(x_k),\\
		x_{k+1} = {}&x_k+\sqrt{s}  \beta_{\rm hb} v_k - s \nabla F(x_k) .
	\end{aligned}
	\right.
	\]
	This gives \cref{eq:hbv-x} and finishes the proof of this lemma.
\end{proof}
\begin{thm}\label{thm:hb-ode}
	Assume $\beta_{\rm hb} = \beta_{\rm hb}(\sqrt{s})$ is a smooth function such that $\lim_{s\to 0}\beta_{\rm hb}(\sqrt{s}) =1$. Then the $O(\sqrt{s})$-resolution ODE of the HB iteration \cref{eq:hb-osr} with respect to the equivalent template \cref{eq:hbv-x} is given by 
	\begin{equation}\label{eq:osr-hb}\small
		\begin{bmatrix}
			x\\v
		\end{bmatrix}'= \begin{bmatrix}
			v\\
			\beta_{\rm hb}'(0)v-\nabla F(x)
		\end{bmatrix}+\frac{\sqrt{s}}{2}\begin{bmatrix}
			\beta_{\rm hb}'(0) v-\nabla F(x)\\
			\left(\beta_{\rm hb}''(0)-[\beta_{\rm hb}'(0)]^2\right)v+	\nabla^2F(x)v
			+\beta_{\rm hb}'(0)\nabla F(x)
		\end{bmatrix}.
	\end{equation}
\end{thm}
\begin{proof}
	Consider the the Taylor expansion of 
	\[
	\Phi_{\rm HB}(X,\tau)=\begin{bmatrix}
		x+\tau\left[	\beta_{\rm hb}(\tau) v-\tau\nabla F(x)\right]\\
		\beta_{\rm hb}(\tau) v-\tau\nabla F(x)
	\end{bmatrix}
	\]
	at $\tau = 0$:
	\[
	\Phi_{\rm HB}(X,\tau) =\Phi_0(X)+\tau\Phi_1(X)+ \frac{\tau^2}{2}\Phi_2(X)+o(\tau^2) ,
	\]
	where $	\Phi_0(X)=X$ and 
	\[
	\Phi_1(X)=\begin{bmatrix}
		v\\
		\beta_{\rm hb}'(0)v-\nabla F(x)
	\end{bmatrix},\,\Phi_2(X)=\begin{bmatrix}
		2	\beta_{\rm hb}'(0) v-2\nabla F(x)\\
		\beta_{\rm hb}''(0)v
	\end{bmatrix}.
	\]
	Thanks to \cref{thm:unique-osr-agd}, it follows that $\Gamma_0(X)=H_{1,0}(X)=\Phi_1(X)$ and 
	\[
	\begin{aligned}
		\Gamma_1(X) ={}& \frac{1}{2}\Phi_2(X)-\frac{1}{2}H_{2,0}(X)=
		\frac{1}{2}\Phi_2(X)-\frac{1}{2}\nabla H_{1,0}(X)H_{1,0}(X)\\
				={}&\frac{1}{2}\begin{bmatrix}
					\beta_{\rm hb}'(0) v-\nabla F(x)\\
					\left(\beta_{\rm hb}''(0)-[\beta_{\rm hb}'(0)]^2\right)v+	\nabla^2F(x)v
					+\beta_{\rm hb}'(0)\nabla F(x)
				\end{bmatrix}.
			\end{aligned}
			\]
			This gives \cref{eq:osr-hb} and concludes the proof.
		\end{proof}
		\begin{coro}\label{coro:osr-hb-polyak}
			Assume $F\in\mathcal S_\mu^2(\R^n)$ and consider the HB iteration \cref{eq:hb-osr} with Polyak's choice  (cf.\cref{eq:s-beta-hb})
			\begin{equation}\label{eq:beta-hb-polyak}
				\beta_{\rm hb}=(1-\sqrt{\mu s})^2.
			\end{equation}
			Then the $O(\sqrt{s})$-resolution ODE with respect to the equivalent template \cref{eq:hbv-x} is  given by
			\begin{equation}\label{eq:osr-hb-polyak}
				\begin{bmatrix}
					x\\v
				\end{bmatrix}'= \begin{bmatrix}
					v\\
					-2\sqrt{\mu}v-\nabla F(x)
				\end{bmatrix}+\frac{\sqrt{s}}{2}\begin{bmatrix}
					-2\sqrt{\mu} v-\nabla F(x)\\
					-2\mu v+	\nabla^2F(x)v
					-2\sqrt{\mu}\nabla F(x)
				\end{bmatrix}.
			\end{equation}
		\end{coro}
		\begin{proof}
			Notice that for $	 \beta_{\rm hb}(\tau)=(1-\sqrt{\mu}\tau)^2$, we have $\beta_{\rm hb}'(0)=-2\sqrt{\mu}$ and $\beta_{\rm hb}''(0) = 2\mu$. Applying \cref{thm:hb-ode} to this setting gives \cref{eq:osr-hb-polyak} and completes the proof.
		\end{proof}
		\begin{coro}\label{coro:osr-hb-nagsc}
			Assume $F\in\mathcal S_\mu^2(\R^n)$ and consider the HB iteration \cref{eq:hb-osr} with the alternate choice \cref{eq:s-beta-shi}.
			Then the $O(\sqrt{s})$-resolution ODE with respect to the equivalent template \cref{eq:hbv-x} is  given by
			\begin{equation}\label{eq:osr-hb-nagsc}
				\begin{bmatrix}
					x\\v
				\end{bmatrix}'= \begin{bmatrix}
					v\\
					-2\sqrt{\mu}v-\nabla F(x)
				\end{bmatrix}+\frac{\sqrt{s}}{2}\begin{bmatrix}
					-2\sqrt{\mu} v-\nabla F(x)\\
					\nabla^2F(x)v
					-2\sqrt{\mu}\nabla F(x)
				\end{bmatrix}.
			\end{equation}
		\end{coro}
		\begin{proof}
			From \cref{eq:s-beta-shi}, we have $	 \beta_{\rm hb}(\tau)=(1-\sqrt{\mu}\tau)/(1+\sqrt{\mu}\tau)$, and it follows that $\beta_{\rm hb}'(0)=-2\sqrt{\mu}$ and $\beta_{\rm hb}''(0) = 4\mu$. Again, invoking \cref{thm:hb-ode} yields \cref{eq:osr-hb-nagsc} and concludes the proof.
		\end{proof}	
		\begin{rem}\label{rem:hb-os}
			According to  \cref{coro:osr-hb-polyak,coro:osr-hb-nagsc}, the $O(1)$-resolution ODEs of \cref{eq:hb} with \cref{eq:s-beta-shi,eq:beta-hb-polyak} are identical to the well-known low-resolution ODE \cref{eq:hb-o1}. However, the $O(\sqrt{s})$-resolution ODEs are
			\begin{equation}\label{eq:osr-hb-nagsc-x}
				x'' +2\sqrt{\mu}  x' + \left(1 + \sqrt{ \mu s } -  \mu s \right) \nabla F(x) - \frac{s}{4} \nabla^2 F(x) \nabla F(x) =0,
			\end{equation}
			and
			\begin{equation}\label{eq:osr-hb-polyak-x}
				x'' +\left(2\sqrt{\mu} + \mu\sqrt{s}\right) x' + \left(1 + \sqrt{ \mu s } -  \mu s /2\right) \nabla F(x) - \frac{s}{4} \nabla^2 F(x) \nabla F(x)=0.
			\end{equation} 
			Compared with the high-resolution model \cref{eq:hb-os-shi} derived by Shi et al. \cite{shi_understanding_2021}, our $O(\sqrt{s})$-resolution ODE \cref{eq:osr-hb-nagsc-x} contains additional high-order $O(s)$-terms and provides a better approximation to \cref{eq:hb}; see \cref{fig:HB-NAG-ode,tab:hb_comparison}.  
		\end{rem}
		\subsection{Analysis of Nesterov's accelerated gradient}
		\label{sec:osr-nag}
		Note that \cref{eq:nag} can be recast into a more general form
		\begin{equation}\label{eq:general-nag}
			x_{k+1} = x_k-s\nabla F(x_k)+\beta_{\rm nag}(x_k-x_{k-1}) - \big[\delta\nabla F(x_k)-\eta\nabla F(x_{k-1})\big],
		\end{equation}
		which contains also the triple-momentum method \cite{fu_2023_understanding} with proper parameters $\delta$ and $\eta$.  
		\begin{lem}\label{lem:general-ang}
			The accelerated gradient method \cref{eq:general-nag} is equivalent to  
			\begin{equation}\label{eq:nagv-x}
				X_{k+1} = \Phi_{\rm NAG}(X_k,\sqrt{s}),
			\end{equation}
			where $X_k=(x_k,v_k)$ and the mapping $\Phi:\R^n\times\R^n\times\R_+\to\R^n\times\R^n$ is defined by
			\begin{equation}\label{eq:phi-nag}
				\Phi_{\rm NAG}(X,\sqrt{s}) := \begin{bmatrix}
					x+\sqrt{s}\beta_{\rm nag}^2v-(\delta+s)\nabla F(x) \\
					\beta_{\rm nag} v-\frac{\delta+s-\eta/\beta_{\rm nag}}{\sqrt{s}\beta_{\rm nag}}\nabla F(x)
				\end{bmatrix},
			\end{equation}
			for all $X= (x,v)\in\R^n\times\R^n$.
			Moreover, if $\beta_{\rm nag} = \beta_{\rm nag}(\sqrt{s}),\,\eta=\eta(\sqrt{s})$ and $\delta=\delta(\sqrt{s})$ are smooth functions such that 
			\begin{equation}\label{eq:cond-nag}
				\lim_{s\to 0}\beta_{\rm nag}(\sqrt{s}) =1,\quad\lim_{s\to 0}\delta(\sqrt{s}) =0,\quad\lim_{s\to 0}\frac{\delta(\sqrt{s})-\eta(\sqrt{s})/\beta_{\rm nag}(\sqrt{s})}{\sqrt{s}\beta_{\rm nag}(\sqrt{s})} =0,
			\end{equation}
			then $\Phi_{\rm NAG}(X,0) = X$ for all $X\in\R^n\times\R^n$.
		\end{lem} 
		\begin{proof}
			Introduce an auxiliary variable
			\begin{equation}\label{eq:vk-nag}
				v_k = \frac{x_k - x_{k-1} + \eta/\beta_{\rm nag}\nabla F(x_{k-1})}{\sqrt{s}\beta_{\rm nag}}.
			\end{equation}
			Then the update $x_{k+1}$ satisfies
			\[
			x_{k+1} = x_k + \sqrt{s}\beta_{\rm nag} v_{k+1} - \eta/\beta_{\rm nag}\nabla F(x_k).
			\]
			Now rearrange  \cref{eq:general-nag} as follows 
			\[
			\begin{aligned}
				x_{k+1} - x_k + \eta/\beta_{\rm nag}\nabla F(x_k) = \beta_{\rm nag}\left[x_k - x_{k-1}+\eta/\beta_{\rm nag}\nabla F(x_{k-1})\right] - (\delta+s-\eta/\beta_{\rm nag})\nabla F(x_k)  ,
			\end{aligned}
			\]
			which implies
			\[
			v_{k+1} = \beta_{\rm nag} v_k - \frac{\delta+s-\eta/\beta_{\rm nag}}{\sqrt{s}\beta_{\rm nag}}\nabla F(x_k).
			\]
			Consequently, we obtain \cref{eq:nagv-x}. As it is easy to check $\Phi_{\rm NAG}(X,0)=X$ with the additional condition \cref{eq:cond-nag}, we complete the proof of this lemma.
		\end{proof}
		\begin{thm}\label{thm:agd}
			Assume that $\beta_{\rm nag} = \beta_{\rm nag}(\sqrt{s}),\,\eta=\eta(\sqrt{s})$ and $\delta=\delta(\sqrt{s})$ are smooth functions satisfying \cref{eq:cond-nag}. Suppose that we have  $\eta(\tau)=\beta_{\rm nag}(\tau)(\tau\eta_1+\tau^2\eta_2+o(\tau^2))$ and 
			\[
			a(\tau):=	\frac{\delta(\tau)+\tau^2-\eta(\tau)/\beta_{\rm nag}(\tau)}{\tau\beta_{\rm nag}(\tau)} =  \tau\delta_1+\tau^2\delta_2+o(\tau^2),
			\]
			with some $\eta_1,\,\eta_2,\,\delta_1,\,\delta_2\in\R$. Then the $O(\sqrt{s})$-resolution ODE of the accelerated gradient method \cref{eq:general-nag} with respect to the equivalent template \cref{eq:nagv-x} is given by 
			\begin{equation}\label{eq:osr-general-nag}
				\begin{bmatrix}
					x\\v
				\end{bmatrix}'= \begin{bmatrix}
					v-\eta_1\nabla F(x)\\
					\beta_{\rm nag}'(0)v-\delta_1\nabla F(x)
				\end{bmatrix}+\frac{\sqrt{s}}{2}\left[\Gamma_{1,1}(x,v)-\Gamma_{1,2}(x,v)\right],
			\end{equation}
			where
			\[
			\begin{aligned}
				\Gamma_{1,1}(x,v):=	{}& \begin{bmatrix}
					4\beta_{\rm nag}'(0)v	-	[2\eta_2+2\delta_1]\nabla F(x)\\
					\beta_{\rm nag}''(0)v-2\delta_2\nabla F(x)
				\end{bmatrix},\\
				\Gamma_{1,2}(x,v):={}& \begin{bmatrix}\eta_1^2\nabla^2F(x)\nabla F(x)
					-\eta_1\nabla^2F(x)v+	\beta_{\rm nag}'(0)v-\delta_1\nabla F(x)\\
					\eta_1\delta_1\nabla^2F(x)\nabla F(x)	-\delta_1\nabla^2F(x)v+[\beta_{\rm nag}'(0)]^2v-\beta_{\rm nag}'(0)\delta_1\nabla F(x)
				\end{bmatrix}.
			\end{aligned}
			\]
		\end{thm}
		\begin{proof}
			Thanks to \cref{lem:general-ang}, NAG \cref{eq:general-nag} is equivalent to $X_{k+1}=\Phi_{\rm NAG}(X_k,\sqrt{s})$ (cf.\cref{eq:nagv-x,eq:phi-nag}). Consider the the Taylor expansion of 
					\[
					\Phi_{\rm NAG}(X,\tau)= \begin{bmatrix}
						x+\tau\beta_{\rm nag}^2(\tau)v-(\delta(\tau)+\tau^2)\nabla F(x) \\
						\beta_{\rm nag}(\tau) v-a(\tau)\nabla F(x)
					\end{bmatrix}
					\]
					at $\tau = 0$:
					\[
					\Phi_{\rm NAG}(X,\tau) =\Phi_0(X)+\tau\Phi_1(X)+ \frac{\tau^2}{2}\Phi_2(X)+o(\tau^2) ,
					\]
					where $	\Phi_0(X)=X$ and 
					\[
					\Phi_1(X)= \begin{bmatrix}
						v-\eta_1\nabla F(x)\\
						\beta_{\rm nag}'(0)v-\delta_1\nabla F(x)
					\end{bmatrix},\,\Phi_2(X)=\begin{bmatrix}
						4\beta_{\rm nag}'(0)v	-	[2\eta_2+2\delta_1]\nabla F(x)\\
						\beta_{\rm nag}''(0)v-2\delta_2\nabla F(x)
					\end{bmatrix}.
					\]
					Thanks to \cref{thm:unique-osr-agd}, it follows that $\Gamma_0(X)=H_{1,0}(X)=\Phi_1(X)$ and 
					\[\small
					\begin{aligned}
						{}&				\Gamma_1(X) = \frac{1}{2}\Phi_2(X)-\frac{1}{2}H_{2,0}(X)=
						\frac{1}{2}\Phi_2(X)-\frac{1}{2}\nabla H_{1,0}(X)H_{1,0}(X)\\
						={}&\frac12\begin{bmatrix}
							4\beta_{\rm nag}'(0)v	-	[2\eta_2+2\delta_1]\nabla F(x)\\
							\beta_{\rm nag}''(0)v-2\delta_2\nabla F(x)
						\end{bmatrix}-\frac{1}{2}\begin{bmatrix}
							-\eta_1\nabla^2F(x)&I\\-\delta_1\nabla^2F(x)&\beta_{\rm nag}'(0)I
						\end{bmatrix}\begin{bmatrix}
							v-\eta_1\nabla F(x)\\
							\beta_{\rm nag}'(0)v-\delta_1\nabla F(x)
						\end{bmatrix}\\				
						={}&\frac{1}{2}\begin{bmatrix}
							3\beta_{\rm nag}'(0) v+\eta_1\nabla^2F(x)v-[2\eta_2+\delta_1]\nabla F(x)-\eta_1^2\nabla^2F(x)\nabla F(x)
							\\
							\left(\beta_{\rm nag}''(0)-[\beta_{\rm nag}'(0)]^2\right)v+\delta_1	\nabla^2F(x)v
							+[\delta_1\beta_{\rm nag}'(0)-2\delta_2]\nabla F(x)-\eta_1\delta_1\nabla^2F(x)\nabla F(x)
						\end{bmatrix}.
					\end{aligned}
					\]
					This leads to \cref{eq:osr-general-nag} and finishes the proof.
				\end{proof}
				\begin{coro}\label{coro:osr-nag-mu}
					Assume $F\in\mathcal S_\mu^2(\R^n)$. Then \cref{eq:nag} with Nesterov's choice (cf.\cref{eq:s-beta-shi})
					\begin{equation}\label{eq:beta-nes}
						\beta_{\rm nag\text{-}sc}=\frac{1-\sqrt{\mu s}}{1+\sqrt{\mu s}}
					\end{equation}
					is equivalent to the template 
					\begin{equation}\label{eq:nag-sc-equi}
						\begin{bmatrix}
							x_{k+1}\\v_{k+1}
						\end{bmatrix}=\begin{bmatrix}
							x_k+\sqrt{s}\beta_{\rm nag\text{-}sc}^2v_k-s(1+\beta_{\rm nag\text{-}sc})\nabla F(x_k) \\
							\beta_{\rm nag\text{-}sc} v_k-\sqrt{s}\nabla F(x_k)
						\end{bmatrix},
					\end{equation}
					and  the $O(\sqrt{s})$-resolution ODE 	is  given by		
					\begin{equation}\label{eq:osr-nag-mu}
						\begin{bmatrix}
							x\\v
						\end{bmatrix}'= \begin{bmatrix}
							v\\
							-2\sqrt{\mu}v- \nabla F(x)
						\end{bmatrix}+\frac{\sqrt{s}}{2}\begin{bmatrix}	-6\sqrt{\mu}v-3\nabla F(x)\\
							\nabla^2F(x)v					-2\sqrt{\mu}\nabla F(x)
						\end{bmatrix}.
					\end{equation} 
				\end{coro}
				\begin{proof}
					Rewrite \cref{eq:nag} as the standard form of \cref{eq:general-nag}:
					\begin{equation}\label{eq:nag-sc-equiv}
						x_{k+1}=x_k-s \nabla F\left(x_k\right)+\beta_{\rm nag\text{-}sc}\left(x_k-x_{k-1}\right)
						-s\beta_{\rm nag\text{-}sc}\left[  \nabla F\left(x_k\right)- \nabla F\left(x_{k-1}\right)\right],
					\end{equation}
					which yields that $\delta=\eta=s\beta_{\rm nag\text{-}sc}$.
					Applying \cref{lem:general-ang} gives the equivalent template \cref{eq:nag-sc-equi}. Clearly, we have  $\beta_{\rm nag\text{-}sc}(\tau) = (1-\sqrt{\mu}\tau)/(1+\sqrt{\mu}\tau)$ and $\beta_{\rm nag\text{-}sc}'(0)=-2\sqrt{\mu},\,\beta_{\rm nag\text{-}sc}''(0)=4\mu$. Moreover, in this case, \cref{eq:cond-nag} holds true and
					\[
					\eta(\tau)=\tau^2\beta_{\rm nag\text{-}sc}(\tau),\quad 		\frac{\delta(\tau)+\tau^2-\eta(\tau)/\beta_{\rm nag\text{-}sc}(\tau)}{\tau\beta_{\rm nag\text{-}sc}(\tau)} =\tau,
					\]
					which implies $\eta_1=\delta_2=0,\,\delta_1=\eta_2=1$.
					Applying  \cref{thm:agd} gives \cref{eq:osr-nag-mu} and concludes the proof.
				\end{proof}
				\begin{rem}\label{rem:compare}
					From \cref{coro:osr-nag-mu}, the $O(1)$-resolution ODE of \cref{eq:nag} with \cref{eq:beta-nes} coincides with the low-resolution ODE \cref{eq:hb-o1} of \cref{eq:hb}, and the $O(\sqrt{s})$-resolution ODE reads as 
					\begin{equation}\label{eq:os1-nag-sc}
						x''+\left(2\sqrt{\mu}+\sqrt{s}\nabla^2 F(x)\right)x'+(1+\sqrt{\mu s}-3\mu s)\nabla F(x) - \frac{3s}{4}\nabla^2F(x)\nabla F(x)=0,
					\end{equation}
					which is very close to the high-resolution ODE \cref{eq:os-nag-sc}, differing from the high-order $O(s)$-terms. As we can see, both \cref{eq:os-nag-sc,eq:os1-nag-sc} have the Hessian-driven damping term $\sqrt{s}\nabla^2F(x)x'$, which, however, does not exist in the $O(\sqrt{s})$-resolution ODEs \cref{eq:osr-hb-nagsc-x,eq:osr-hb-polyak-x} of \cref{eq:hb}.
					
					In view of \cref{eq:beta-hb-polyak,eq:beta-nes}, we have $\beta_{\rm hb}=\beta_{\rm nag\text{-}sc}(1-\mu s) $ and reformulate \cref{eq:hb} as 
					\begin{equation}\label{eq:hb-equi}
						x_{k+1}=x_k-s \nabla F\left(x_k\right)+\beta_{\rm nag\text{-}sc}\left(x_k-x_{k-1}\right)
						- \mu s \beta_{\rm nag\text{-}sc}(x_k-x_{k-1}).
					\end{equation}
					Observing \cref{eq:nag-sc-equiv}, we find that the subtle difference comes from $	 \mu s \beta_{\rm nag\text{-}sc}(x_k-x_{k-1})$ and 
					$s\beta_{\rm nag\text{-}sc}\left[ \nabla F\left(x_k\right)- \nabla F\left(x_{k-1}\right)\right]$, both of which are high-order $O(s)$-terms. The former refers to the velocity correction $\mu\sqrt{s}x'$ in \cref{eq:osr-hb-polyak-x} while the latter is called the gradient correction (cf.\cite[Section 1.1]{shi_understanding_2021}), also known as the Hessian-driven damping term $\sqrt{s}\nabla^2F(x)x'$. 
				\end{rem}
				
				In \cref{fig:HB-NAG-ode}, we report the trajectories of the low and high-resolution ODEs of \cref{eq:hb,eq:nag}. As expected, the $O(\sqrt{s})$-resolution ODEs are much better close to the discrete methods than the $O(1)$-resolution ODEs. Also, the high-resolution ODEs \cref{eq:hb-os-shi,eq:os-nag-sc} does not provide approximations as good as our $O(\sqrt{s})$-resolution models.
				\begin{figure}[H]
					\centering  
					{
						\begin{minipage}[t]{0.35\textwidth}
							\centering          
							\includegraphics[width=0.9\textwidth]{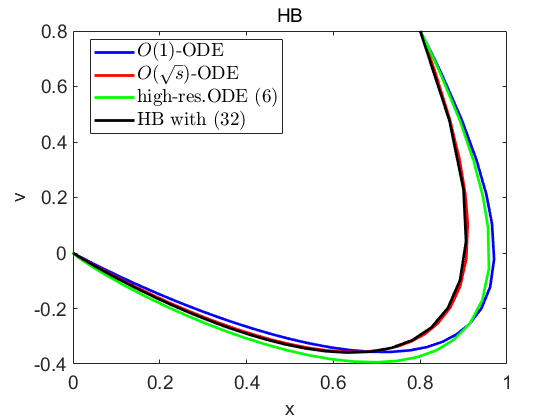}   
						\end{minipage}%
					}\centering  
					{
						\begin{minipage}[t]{0.35\textwidth}
							\centering          
							\includegraphics[width=0.9\textwidth]{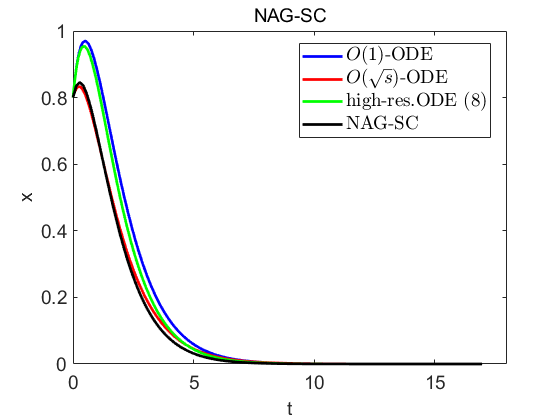}   
						\end{minipage}%
					}
					\caption{Illustration of the behaviors of the discrete-time algorithms and their corresponding ODEs. The objective is $F(x) = 1/2x^2$ with the step-size $s=0.02$ and initial condition $(x_0,v_0)=(0.8, 0.8)$.} 
					\label{fig:HB-NAG-ode}  
				\end{figure}
				
				\begin{table}[H]
					\centering
					\renewcommand{\arraystretch}{0.75}	
					\caption{Convergence rates of the $O(1)$-resolution ODEs of HB and NAG-SC}
					\label{tab:hb_comparison}
					\begin{tabular}{c c S[table-format=1.2e-2] S[table-format=1.2] S[table-format=1.2] S[table-format=1.2] S[table-format=1.2e-2] S[table-format=1.2]}
						\toprule
						{Algor. / $O(1)$} & {$s$} & {$E_1(s)$} & {Rate} & {$E_2(s)$} & {Rate} & {$E_3(s)$} & {Rate} \\
						\midrule
						\multirow{5}{*}{HB / \text{\cref{eq:hb-o1}}}
						& $1/2^{4}$ & 5.81e-03 & {--} & 3.43e+00 & {--} & 9.65e-02 & {--} \\
						& $1/2^{5}$ & 1.46e-03 & 1.99 & 3.47e+00 & -0.01 & 4.85e-02 & 0.99 \\
						& $1/2^{6}$ & 3.66e-04 & 2.00 & 3.48e+00 & -0.01 & 2.43e-02 & 1.00 \\
						& $1/2^{7}$ & 9.15e-05 & 2.00 & 3.49e+00 & -0.00 & 1.22e-02 & 1.00 \\
						\midrule
						\multirow{5}{*}{HB / \text{\cref{eq:hb-os-shi}}}
						& $1/2^{4}$ & 5.48e-03 & {--} & 1.44e+00 & {--} & 4.62e-02 & {--} \\
						& $1/2^{5}$ & 1.42e-03 & 1.95 & 1.47e+00 & -0.03 & 2.33e-02 & 0.99 \\
						& $1/2^{6}$ & 3.60e-04 & 1.98 & 1.48e+00 & -0.02 & 1.17e-02 & 0.99 \\
						& $1/2^{7}$ & 9.08e-05 & 1.99 & 1.49e+00 & -0.01 & 5.85e-03 & 1.00 \\
						\midrule
						\multirow{5}{*}{NAG-SC / \text{\cref{eq:hb-o1}}}
						& $1/2^{4}$ & 1.55e-02 & {--} & 5.92e+00 & {--} & 1.74e-01 & {--} \\		
						& $1/2^{5}$ & 4.13e-03 & 1.91 & 6.20e+00 & -0.07 & 9.06e-02 & 0.94 \\
						& $1/2^{6}$ & 1.06e-03 & 1.96 & 6.35e+00 & -0.03 & 4.63e-02 & 0.97 \\
						& $1/2^{7}$ & 2.70e-04 & 1.98 & 6.42e+00 & -0.02 & 2.34e-02 & 0.98 \\
						\bottomrule
					\end{tabular}
				\end{table}
				\begin{table}[H]
					\centering
					\setlength{\tabcolsep}{3.2pt}
					\renewcommand{\arraystretch}{0.8}	
					\caption{Convergence rates of the $O(\sqrt{s})$-resolution ODEs of HB and NAG-SC }
					\label{tab:hb_os}
					\begin{tabular}{c c S[table-format=1.2e-2] S[table-format=1.2] S[table-format=1.2e-2] S[table-format=1.2] S[table-format=1.2e-2] S[table-format=1.2]}
						\toprule
						{Algor. / $O(\sqrt{s})$} & {$s$} & {$E_1(s)$} & {Rate} & {$E_2(s)$} & {Rate} & {$E_3(s)$} & {Rate} \\
						\midrule
						\multirow{5}{*}{HB / \text{\cref{eq:osr-hb-polyak-x}}}
						& $1/2^{4}$ & 9.72e-05 & {--} & 2.62e-01 & {--} & 6.77e-03 & {--} \\
						& $1/2^{5}$ & 1.12e-05 & 3.12 & 1.38e-01 & 0.92 & 1.77e-03 & 1.94 \\
						& $1/2^{6}$ & 1.34e-06 & 3.07 & 7.09e-02 & 0.96 & 4.52e-04 & 1.97 \\
						& $1/2^{7}$ & 1.63e-07 & 3.03 & 3.60e-02 & 0.98 & 1.14e-04 & 1.98 \\
						\midrule
						\multirow{5}{*}{NAG-SC / \text{\cref{eq:os1-nag-sc}}}
						& $1/2^{4}$ & 1.15e-03 & {--} & 4.98e-01 & {--} & 1.22e-02 & {--} \\
						& $1/2^{5}$ & 1.58e-04 & 2.86 & 2.73e-01 & 0.87 & 3.33e-03 & 1.87 \\
						& $1/2^{6}$ & 2.06e-05 & 2.93 & 1.44e-01 & 0.93 & 8.72e-04 & 1.93 \\
						& $1/2^{7}$ & 2.64e-06 & 2.97 & 7.37e-02 & 0.96 & 2.23e-04 & 1.97 \\
						\bottomrule
					\end{tabular}
				\end{table} 
				
				Again, let us check the convergence rate of the $O((\sqrt{s})^{r})$-resolution ODEs regarding to three measurements:
				\[
				E_1(s):=\nm{x(\sqrt{s})-x_1},\, E_2(s):=\sum_{k=1}^{N}\nm{x(t_k)-x_k},\, 
				E_3(s):=\sqrt{\sum_{k=1}^{N}\sqrt{s}\nm{x(t_k)-x_k}^2},
				\]
				where $t_k = k\sqrt{s}$ for $1\leq k\leq N=T/\sqrt{s}$ with fixed time $T>0$. Here, we focus only on the component $x$ not the whole vector $X$ because existing ODEs do not admit proper first-order presentations like \cref{eq:osr-hb-polyak}. According to \cref{rem:0ss-bound}, if $X(0)=X_0$, then we have $E_1(s)=O(s^{\frac{r+2}{2}}),E_2(s)=O(s^{\frac{r}{2}})$ and $E_3(s)=O(s^{\frac{r+1}{2}})$. This agrees well with  which are verified by the numerical results in \cref{tab:pdhg_comparison,tab:pdhg_osarison_poly}. 

					\subsection{Analysis of accelerated gradient methods with variable parameters}
					In this section, we focus on accelerated gradient methods with variable parameters. Following the main idea from \cite{luo_icpdps_2025}, to find a proper equivalent template, we aim to seek the intrinsic finite difference presentation in terms of the intrinsic step size $\sqrt{s}$.
					
					\begin{thm}\label{thm:osr-nag-c}
						Assume $F\in\mathcal F^2(\R^n)$. Then \cref{eq:nag} with the dynamical changing parameter
						\begin{equation}\label{eq:beta-nes-c}
							\beta_{\rm nag\text{-}c}=\frac{k}{k+3}
						\end{equation}
						is equivalent to  $ 	X_{k+1} = \Phi_{\rm nag\text{-}c}(X_k,\sqrt{s})$, where $X_k=(x_k,v_k,t_k)$ and 
						\begin{equation}\label{eq:Phi-nag-c}
							\Phi_{\rm nag\text{-}c}(X,\tau):=\begin{bmatrix}
								x - \tau^2\left(2 - 3\tau /t\right)\nabla F(x)+ \tau\left(1 - 3\tau /t\right)^2v \\
								\left(1 - 3\tau/t\right)v- \tau\nabla F(x)\\
								t+\tau
							\end{bmatrix},\,X = (x,v,t).
						\end{equation}
						Moreover, the corresponding $O(\sqrt{s})$-resolution ODE
						is given by
						\begin{equation} 
							\label{eq:osr-nag-c}
							\begin{bmatrix}
								x\\v\\ t
							\end{bmatrix}'= \begin{bmatrix}
								v \\-3/tv-\nabla F(x)\\1
							\end{bmatrix} 
							+\frac{\sqrt{s}}{2}\begin{bmatrix}
								-9/tv-3\nabla F(x)\\\nabla^2 F(x)v-12/t^2v-3/t\nabla F(x)\\0
							\end{bmatrix}.
						\end{equation} 
					\end{thm}
					\begin{proof}
						It is not hard to verify the equivalent template \cref{eq:Phi-nag-c} by letting $t_k=(k+3)\sqrt{s}$ and $v_k = \left[x_k-x_{k-1}+s\nabla F(x_{k-1})\right]/(\sqrt{s}-3s/t_k)$. It is clear that $\Phi_{\rm nag\text{-}c}(X,0) = X$. 	Consider the the Taylor expansion of $\Phi_{\rm nag\text{-}c}(X,\tau)$
						at $\tau = 0$:
						\[
						\Phi_{\rm nag\text{-}c}(X,\tau)=\Phi_0(X)+\tau\Phi_1(X)+ \frac{\tau^2}{2}\Phi_2(X)+o(\tau^2) ,
						\]
						where $	\Phi_0(X)=X$ and 
						\[
						\Phi_1(X)= \begin{bmatrix}  
							v \\-3/tv-\nabla F(x)\\1
						\end{bmatrix},\,\Phi_2(X)=\begin{bmatrix}
							-4\nabla F(x)-12/tv\\0\\0 
						\end{bmatrix}.
						\]
						Thanks to \cref{thm:unique-osr-agd}, it follows that $\Gamma_0(X)=H_{1,0}(X)=\Phi_1(X)$ and 
						\[\small
						\begin{aligned}
							{}&				\Gamma_1(X) = \frac{1}{2}\Phi_2(X)-\frac{1}{2}H_{2,0}(X)=
							\frac{1}{2}\Phi_2(X)-\frac{1}{2}\nabla H_{1,0}(X)H_{1,0}(X)\\
							={}&\frac12\begin{bmatrix}
								-4\nabla F(x)-12/tv\\0\\0
							\end{bmatrix}-\frac{1}{2}\begin{bmatrix}
								O&I&O\\- \nabla^2F(x)&-3/tI&3/t^2v\\O&O&O
							\end{bmatrix}\begin{bmatrix}
								v \\-3/tv-\nabla F(x)\\1
							\end{bmatrix}\\
								={}&\frac{1}{2}\begin{bmatrix}
									-9/tv-3\nabla F(x)\\\nabla^2 F(x)v-12/t^2v-3/t\nabla F(x)\\0
								\end{bmatrix}.
							\end{aligned}
							\]
							This leads to \cref{eq:osr-nag-c}                                                                                                                                                                                                                                                                                                                                                                                                                                                                                                                                                                                                                                                                                                                                                                                                                                                                                                                                                                                                                                                                                                                                                                                                                                                                                                                                                                                                                                                                                                                                                                                                                                                                                                                                                                                                                                                                                                                                                                                                                                                                                                                                                                                                                                                                                                                                                                                                                                                                                                                                                                                                                                                                                                                                                                                                                                                                                                                                                                                                                                                                                                                                                                                                                                                                                                                                                                                                                                                                                                                                                                                                                                                                                                                                                                                                                                                                                                                                                                                      and finishes the proof.
						\end{proof}
						
						\begin{rem}
							Thanks to \cref{thm:osr-nag-c}, the $O(1)$-resolution ODE of NAG-C coincides with the low-resolution model \cref{eq:avd}, and the $O(\sqrt{s})$-resolution ODE is  
							\begin{equation} 
								\label{eq:os1-nag-c}\small
								x''  +\left( \frac{3}{t}+\frac{6\sqrt{s}}{t^2} +\sqrt{s}\nabla^2F(x)\right) x' 
								+ \left(1 + \frac{3\sqrt{s}}{2t}+\frac{9s}{4t^2}\right)\nabla F(x)  - \dfrac{3s}{4} \nabla^2 F(x) \nabla F(x)= 0.
							\end{equation} 
							Recall the high-resolution ODE \cref{eq:os-ode-nag-c} derived in \cite[Eq.(1.12)]{shi_understanding_2021}:
							\begin{equation}\label{eq:os1-nag-c-shi}
								x'' + \left( \frac{3}{t} + \sqrt{s} \nabla^2 F(x) \right) x' 
								+ \left(1 + \frac{3\sqrt{s}}{2t} \right) \nabla F(x) = 0.
							\end{equation}
							The difference is  $\frac{6\sqrt{s}}{t^2}x' + O(s)$. We claim that our model \cref{eq:os1-nag-c} is more accurate than \cref{eq:os1-nag-c-shi}. As $t$ grows, these two ODEs become increasingly aligned; see \cref{fig:nag-c-ode-compare}. 
							\begin{figure}[H]
								\centering  
								{
									\begin{minipage}[t]{0.4\textwidth}
										\centering          
										\includegraphics[width=0.9\textwidth]{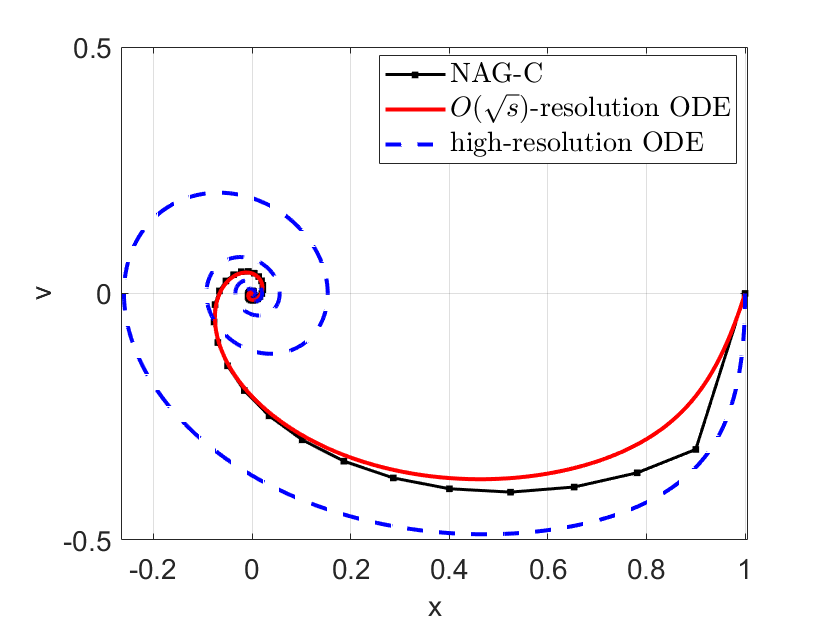}   
									\end{minipage}%
								}
								\centering  
								{
									\begin{minipage}[t]{0.4\textwidth}
										\centering          
										\includegraphics[width=0.9\textwidth]{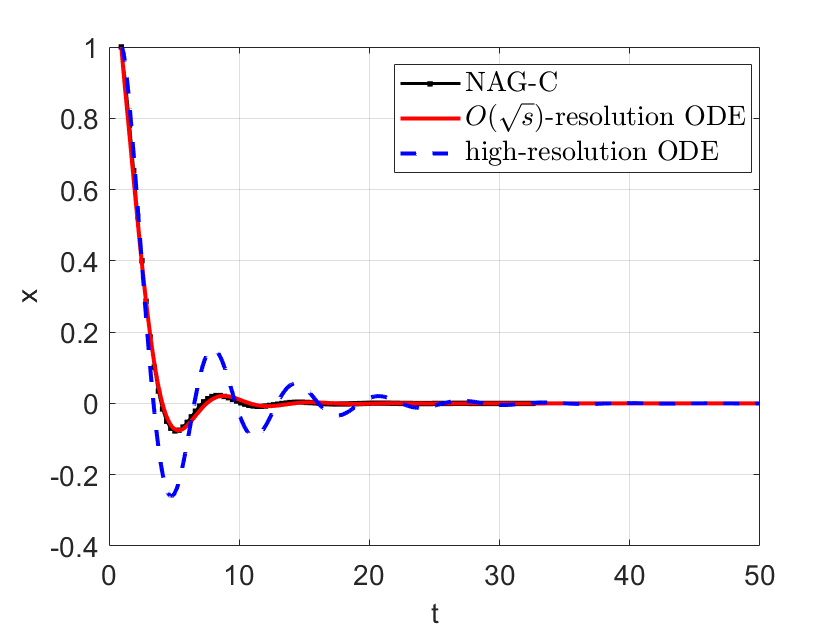}   
									\end{minipage}%
								}
								\caption{Illustration of the behaviors of NAG-C  and  the high-resolution ODEs \cref{eq:os1-nag-c,eq:os1-nag-c-shi} with the step size $s=0.1$. The objective function is  $F(x)=1/2x^2$ for $x\in \R$.} 
								\label{fig:nag-c-ode-compare}  
							\end{figure} 
						\end{rem}
						
						As the final example, let us look at the accelerated mirror descent (AMD) \cite[Eq.(3.11)]{2003Smooth}. Recall that $\varphi\in \mathcal S_1^1(\R^n)$ is a given prox-function.
						\begin{thm}\label{thm:osr-amd}
							Assume $F\in\mathcal F^2(\R^n)$. The $O(\sqrt{s})$-resolution ODE of the	accelerated mirror descent:
							\begin{equation}\label{eq:amd}\tag{AMD}\small
								\left\{
								\begin{aligned}
									&y_{k+1} = \mathop{\argmin}\limits_{y \in \mathbb{R}^d} \left\{ s \langle \nabla F(x_k), y - x_k \rangle + \frac{1}{2} \|y - x_k\|^2 \right\}\\
									&z_{k+1} = \mathop{\argmin}\limits_{z\in \mathbb{R}^d} \left\{ \frac{1}{  s} \varphi(z) + \sum_{i=0}^{k} \frac{i + 1}{2} [ F(x_i) + \langle \nabla F(x_i), z - x_i \rangle ] \right\} \\
									&x_{k+1} = \frac{2}{k + 3} z_{k+1} + \frac{k + 1}{k + 3} y_{k+1}
								\end{aligned}
								\right.
							\end{equation}
							with respect to the equivalent template $ 	X_{k+1} = \Phi_{\rm amd}(X_k,\sqrt{s})$ where $X_k=(x_k,z_k,t_k)$ and 
							\begin{equation}\label{eq:Phi-amd}\small
								\Phi_{\rm amd}(X,\tau):=\begin{bmatrix}
									(1-\frac{2\tau}{t})x+\frac{2\tau}{t}\nabla \varphi^*(z -  \tau (t/2-\tau) \nabla F(x)) -\tau^2(1-\frac{2\tau}{t})\nabla F(x)\\
									z -  \tau (t/2-\tau) \nabla F(x)\\
									t+\tau 
								\end{bmatrix}
							\end{equation} 
							is given by
							\begin{equation} \small
								\label{eq:osr-amd}
								\begin{bmatrix}
									x\\z\\ t
								\end{bmatrix}'=\begin{bmatrix}  
									\frac{2}{t}(\nabla\varphi^*(z)-x)\\-  t/2\nabla F(x)\\1
								\end{bmatrix}+\frac{\sqrt{s}}{2}\begin{bmatrix}
									-  \nabla^2\varphi^*(z)\nabla F(x)-2\nabla F(x)+6/t^2(\nabla \varphi^*(z)-x)\\\frac{5 }{2} \nabla F(x)+ \nabla^2F(x)(\nabla\varphi^*(z)-x)\\0
								\end{bmatrix}.
							\end{equation} 
						\end{thm}
						\begin{proof}
							Following \cite[Section 3.1]{Yuan2024AnalyzeAM}, it has been proved that
							\begin{equation}
								\left\{
								\begin{aligned}
									&z_{k+1}  =z_k -  s \frac{k + 1}{2} \nabla F(x_k), \\
									&x_{k+1} =\frac{2}{k+3} \nabla \varphi^*(z_{k+1}) + \frac{k+1}{k+3}  \left[ x_k - s \nabla F(x_k) \right].
								\end{aligned}
								\right.\nonumber
							\end{equation}
							Setting $t_k=(k+3)\sqrt{s}$ yields the equivalent form \cref{eq:Phi-amd}. Note that $\Phi_{\rm amd}(X,0) = X$.
							Consider the the Taylor expansion of $\Phi_{\rm amd}(X,\tau)$
							at $\tau = 0$:
							\[
							\Phi_{\rm amd}(X,\tau)=\Phi_0(X)+\tau\Phi_1(X)+ \frac{\tau^2}{2}\Phi_2(X)+o(\tau^2) ,
							\]
							where $	\Phi_0(X)=X$ and 
							\[
							\Phi_1(X)= \begin{bmatrix}  
								2/t\nabla \varphi^*(z)-2/ax \\-  t/2\nabla F(x)\\1
							\end{bmatrix},\,\Phi_2(X)=\begin{bmatrix}
								-2  \nabla^2\varphi^*(z)\nabla F(x)-2\nabla F(x)\\  \nabla F(x)\\0
							\end{bmatrix}.
							\]
							Thanks to \cref{thm:unique-osr-agd}, it follows that $\Gamma_0(X)=H_{1,0}(X)=\Phi_1(X)$ and 
							\[\small
							\begin{aligned}
								{}&				\Gamma_1(X) = \frac{1}{2}\Phi_2(X)-\frac{1}{2}H_{2,0}(X)=
								\frac{1}{2}\Phi_2(X)-\frac{1}{2}\nabla H_{1,0}(X)H_{1,0}(X)\\ 
								={}&\frac{1}{2}\Phi_2(X)-\frac{1}{2}\begin{bmatrix}
									-\frac{2}{t}&\frac{2}{t}\nabla^2\varphi^*(z)&-\frac{2}{t^2}(\nabla\varphi^*(z)-x)\\ -\frac{  t}{2}\nabla^2F(x)&O&-\frac{1}{2}\nabla F(x)\\O&O&O
								\end{bmatrix}\begin{bmatrix}
									\frac{2}{t}(\nabla \varphi^*(z)-x) \\-  t/2\nabla F(x)\\1
								\end{bmatrix}\\
								={}&\frac{1}{2}\begin{bmatrix}
									-  \nabla^2\varphi^*(z)\nabla F(x)-2\nabla F(x)+6/t^2(\nabla \varphi^*(z)-x)\\\frac{5}{2} \nabla F(x)+ \nabla^2F(x)(\nabla\varphi^*(z)-x)\\0
								\end{bmatrix}.
							\end{aligned}
							\]
							This leads to \cref{eq:osr-amd}                                                                                                                                                                                                                                                                                                                                                                                                                                                                                                                                                                                                                                                                                                                                                                                                                                                                                                                                                                                                                                                                                                                                                                                                                                                                                                                                                                                                                                                                                                                                                                                                                                                                                                                                                                                                                                                                                                                                                                                                                                                                                                                                                                                                                                                                                                                                                                                                                                                                                                                                                                                                                                                                                                                                                                                                                                                                                                                                                                                                                                                                                                                                                                                                                                                                                                                                                                                                                                                                                                                                                                                                                                                                                                                                                                                                                                                                                                                                                                                      and finishes the proof.
						\end{proof}
						\begin{rem}
							By \cref{eq:osr-amd}, the $O(1)$-resolution ODE for \cref{eq:amd} reads as 
							\[
							x'   = \frac{2}{t}\left(\nabla \varphi^*(z)-x\right),\quad
							z'  = -\frac{t }{2}\nabla F(x),
							\]
							which coincides with the low-resolution ODE of \cref{eq:amd} derived in \cite{krichene_accelerated_2015}. However, we note that our $O(\sqrt{s})$-resolution ODE \cref{eq:osr-amd} differs from the high-resolution ODE of \cref{eq:amd} in \cite{Yuan2024AnalyzeAM}:
							\begin{equation}
								x'   = \frac{2}{t}\left(\nabla \varphi^*(z)-x\right)- \sqrt{s}\nabla F(x),\quad
								z'  = -\frac{t }{2}\nabla F(x), 
							\end{equation} 
							which is between our $O(1)$-resolution ODE and $O(\sqrt{s})$-resolution ODE.
						\end{rem}

\section{PDHG with $O(s)$-Correction}\label{sec:ode to dta-pdhg}
\subsection{Correction for continuous-time PDHG}
We follow the $O(s)$-correction idea from \cite{lu_osr-resolution_2022} and treat the $O(s)$-resolution ODE \cref{eq:osr-pdhg} of CP as a correction to the $O(1)$-resolution ODE \cref{eq:pdhg-o1}. To avoid the second order derivative $\nabla M(Z)$, we drop the Hessian terms $\nabla^2 F$ and $\nabla^2 G$ and consider the following $O(s)$-correction PDHG ODE:
\begin{equation}\label{eq:os-cor-pdhg}
	Z' =\mathcal G(Z):= -M(Z)+\frac{s}{2}\left[\eta_1 Q+2\eta_2QI_\theta\right]M(Z),
\end{equation}
where $Q$ and $I_\theta$ are defined in \cref{eq:ppa-pdhg} and $\eta_1,\eta_2>0$ are weight parameters. In component wise, letting $Z=(x,y)$ yields that
\begin{equation}\label{eq:CP-G}
	\left\{
	\begin{aligned}
		x'={}&-  \nabla_x\mathcal{L}(x,y)-\frac{s}{2}(\eta_1-2\eta_2 )A^\top\nabla_y\mathcal{L}(x,y) ,\\
		y'={}& \nabla_y\mathcal{L}(x,y)-\frac{s}2(\eta_1+2\theta\eta_2)A\nabla_x\mathcal{L}(x,y). 
	\end{aligned}
	\right.
\end{equation} 
\begin{rem}
	The original $O(s)$-correction idea for GDA by Lu \cite{lu_osr-resolution_2022} utilizes the difference between the $O(s)$-terms of \cref{eq:osr-gda,eq:osr-ppm} and considered the model $Z' = \nabla M(Z)M(Z)$. 
	The explicit discretization $Z_{k+1} = Z_k+s\nabla M(Z_k)M(Z_k)$ is called the Jacobian Method (JM).  Note that our $O(s)$-correction ODE \cref{eq:os-cor-pdhg} is Hessian-free and the discrete scheme \cref{eq:ex-osr-pdhg} does not involve any second-order information either. Furthermore, even if a spectrum argument for the case of a bilinear saddle point function $L(x,y)=y^\top Bx$ has shown the stability and convergence behavior of JM, the convergence for general saddle problems remains unclear.
\end{rem}

Introduce the following Lyapunov function
\begin{equation}\label{eq:lyapunov}
	\mathcal{E}(Z):=\frac{1}{2}\nm{Z-Z^*}^2=\frac{1}{2}\left\| x-x^*\right\| ^2+\frac{1}{2}\left\| y-y^*\right\| ^2,\quad\forall\,Z = (x,y)\in\R^n\times\R^m,
\end{equation} 
where $Z^*=(x^*,y^*)\in M^{-1}(0)$. To establish the convergence rate, we shall verify the strong Lyapunov property. The key is the following lower bound of a cross term.
\begin{lem}\label{lem:f-bound}
	Suppose $F\in\mathcal F_{L_{f}}^1(\R^n)$ and $G\in \mathcal{F}_{L_g}^{1}(\R^m)$. Then for any $R\in\mathbb S_{m+n}^+$, we have 
	\begin{equation}\label{eq:cross-term}
		\begin{aligned}
			{}&\snm{\dual{QR(H(Z_1)-H(Z_2)),Z_1-Z_2}}\\\geq{}& -\frac{\nm{R}}{2s}\dual{M(Z_1)-M(Z_2),Z_1-Z_2}-\frac{sL}{2}\nm{Q(Z_1-Z_2)}_{R}^2,
		\end{aligned}
	\end{equation}
	for all $Z_1,\,Z_2\in\R^{m+n}$, where $L:=\max\{L_f,L_g\}$ and 
	\begin{equation}\label{eq:H}
		H(Z):= \begin{bmatrix}
			\nabla F(x)\\\nabla G(y)
		\end{bmatrix},\quad \forall\,Z = (x,y)\in\R^{m+n}.
	\end{equation}
\end{lem}
\begin{proof}
	Recall that $	M(Z)=H(Z)+QZ$. 
	It follows from the Cauchy-Schwarz inequality  that 
	\[
	\begin{aligned}
		{}&		 	\snm{	\dual{QR(H(Z_1)-H(Z_2)),Z_1-Z_2}}
		=		 	\snm{	\dual{R^{1/2}(H(Z_1)-H(Z_2)),R^{1/2}Q^\top(Z_1-Z_2)}}\\
		\geq {}&-	\frac{1}{2s}\nm{R^{1/2}(H(Z_1)-H(Z_2))}^2_{D^{-1}}-	\frac{s}{2}\nm{R^{1/2}Q^\top(Z_1-Z_2)}_{D}^2\\
		\geq {}&-	\frac{\nm{R}}{2s}\nm{H(Z_1)-H(Z_2)}_{D^{-1}}^2-\frac{s\nm{D}}{2}\nm{Q^\top(Z_1-Z_2)}_{R}^2,		 				 		 		
	\end{aligned}
	\]
	where $D = \diag{L_fI,L_hI}$ and $\nm{D}\leq L$.
	Then using \cref{eq:low-bound} gives
	\begin{equation}\label{eq:H-H}
		\begin{aligned}
			\nm{H(Z_1)-H(Z_2)}_{D^{-1}}^2=	{}&\frac{1}{L_f}\|\nabla F(x_1)-\nabla F(x_2)\|^2 +\frac{1}{L_h}\|\nabla G(y_1)-\nabla G(y_2)\|^2\\
			\leq	{}	& \left\langle \nabla F(x_1)-\nabla F(x_2),x_1-x_2\right\rangle + \left\langle \nabla G(y_1)-\nabla G(y_2),y_1-y_2\right\rangle \\
			={}&\dual{M(Z_1)-M(Z_2),Z_1-Z_2}.
		\end{aligned}
	\end{equation}
	Since $\nm{D}\leq L$ and $Q^\top=-Q$, this leads to \cref{eq:cross-term} and completes the proof.
\end{proof}
We then verify the strong Lyapunov property of $\mathcal E$ with respect to \cref{eq:os-cor-pdhg}.
\begin{lem}\label{thm:con-slp}
	Suppose $F\in\mathcal F_{L_{f}}^1(\R^n)$ and $G\in \mathcal{F}_{L_g}^{1}(\R^m)$. Assume $\theta\geq -1$ and let $\eta_1,\,\eta_2\in\R_+$ be such that $2\eta_2< \eta_1<4-2\theta\eta_2$. If $0< s<2/L$ with $L:=\max\{L_f,L_g\}$, then we have
	\begin{equation}\label{eq:slp-pdhg}
		-\dual{\nabla \mathcal{E}(Z), \mathcal{G}(Z)}\geq{} C_1(\eta_1,\eta_2,\theta)\left\langle M(Z),Z-Z^*\right\rangle +C_2(\eta_1,\eta_2,s)\nm{Q(Z-Z^*)}^2,
	\end{equation}
	for all $Z = (x,y)\in\R^n\times\R^m$, where $C_1(\eta_1,\eta_2,\theta):=\frac{4-\eta_1-2\theta \eta_2}{4}>0$ and $C_2(\eta_1,\eta_2,s):=\frac{s(2-sL)}{4}(\eta_1-2\eta_2)>0$.
\end{lem}
\begin{proof}
	From \cref{eq:os-cor-pdhg,eq:H}, we have
	\begin{equation}\label{eq:decomp-G}
		\mathcal G(Z) = -M(Z)+\frac{s}{2}Q(\eta_1I+2\eta_2I_\theta)M(Z)=
		-M(Z)+\frac{s}{2}QRH(Z)
		-\frac{s}{2}Q^\top RQZ,
	\end{equation}
	where $R=\eta_1I+2\eta_2I_\theta=\diag{(\eta_1+2\theta\eta_2)I,(\eta_1-2\eta_2)I}$.
	Since $Z^*\in M^{-1}(0)$,  it holds that $\mathcal G(Z^*)=0$ and
	\[
	\begin{aligned}
		{}		&-\dual{\nabla \mathcal{E}(Z),\mathcal{G}(Z)}= -\dual{Z-Z^*,\mathcal G(Z)-\mathcal G(Z^*)}\\
		={}&\dual{M(Z),Z-Z^*}+\frac{s}{2}\nm{Q(Z-Z^*)}^2_R-\frac{s}{2}\dual{QR(H(Z)-H(Z^*)),Z-Z^*}.
	\end{aligned}
	\]
	We then apply \cref{lem:f-bound} to  get
	\[\small
	-\frac{s}{2}\dual{QR(H(Z)-H(Z^*)),Z-Z^*}\geq
	-\frac{\nm{R}}{4}\dual{M(Z),Z-Z^*}-\frac{s^2L}{4}\nm{Q(Z-Z^*)}_{R}^2,
	\]
	which implies 
	\[
	\begin{aligned}
		{}&		-\dual{\nabla \mathcal{E}(Z),\mathcal{G}(Z)}
		\geq \frac{4-\nm{R}}{4}\dual{M(Z),Z-Z^*}+\frac{s(2-sL)}{4}\nm{Q(Z-Z^*)}^2_R\\
		\geq {}&\frac{4-\eta_1-2\theta\eta_2}{4}\dual{M(Z),Z-Z^*}+\frac{s(2-sL)}{4}(\eta_1-2\eta_2)\nm{Q(Z-Z^*)}^2.
	\end{aligned}
	\]
	This finishes the proof of this lemma.
\end{proof}
\begin{thm}
	Suppose $F\in\mathcal F_{L_{f}}^1(\R^n)$ and $G\in \mathcal{F}_{L_g}^{1}(\R^m)$. There exists a unique global $C^1$-smooth solution $Z(t) = (x(t),y(t))$ to the $O(s)$-correction PDHG ODE \cref{eq:os-cor-pdhg} with $Z(0)=Z_0=(x_0,y_0)\in\R^{m+n}$. Assume $\theta\geq -1$ and let $\eta_1,\,\eta_2\in\R_+$ be such that $2\eta_2< \eta_1<4-2\theta\eta_2$. If $0< s<2/L$ with $L:=\max\{L_f,L_g\}$, then we have 
	\begin{equation}\label{eq:rate-mod-pdhg-osr}
		\begin{aligned}
			\mathcal E(Z(t))+{}&C_1(\eta_1,\eta_2,\theta)\int_{0}^{t} \left\langle M(Z(r)),Z(r)-Z^*\right\rangle\dd r \\
			{}&\quad +C_2(\eta_1,\eta_2,s)\int_{0}^{t}\nm{Q(Z(r)-Z^*)}^2\dd r\leq \mathcal E(Z(0)),
		\end{aligned}
	\end{equation}
	for all $t>0$, where both $C_1$ and $C_2$ are defined in \cref{eq:slp-pdhg}. As by products, we get the ergodic rates
	\begin{align}
		\nm{Q(\bar{Z}(t)-Z^*)}^2\leq  {}&\frac {\mathcal E(Z_0)+C_2(\eta_1,\eta_2,s)\nm{Q(Z_0-Z^*)}^2}{C_2(\eta_1,\eta_2,s)(1+t)},
		\label{eq:est1-rate-mod-pdhg-osr}\\
		\mathcal L(\bar{x}(t),y^*)-	 \mathcal L(x^*,\bar{y}(t))\leq{}& \frac{\mathcal E(Z_0) +C_1(\eta_1,\eta_2,\theta)\left[\mathcal L(x_0,y^*)-	 \mathcal L(x^*,y_0)\right]}{	C_1(\eta_1,\eta_2,\theta)(1+t)},
		\label{eq:est2-rate-mod-pdhg-osr}
	\end{align}
	where 
	\[
	\begin{aligned}
		\bar{x}(t) :={}&\frac{x_0+\int_{0}^{t}x(r)\dd r}{1+t},\quad 
		\bar{y}(t):={}\frac{y_0+\int_{0}^{t}y(r)\dd r}{1
			+t}.
	\end{aligned}
	\]
	Moreover, if $A$ is invertible, then $\lambda_{\min}(Q^\top Q)=\sigma_{\min}^2(A)>0$ and we have the exponential rate
	\begin{equation}\label{eq:exp-pdhg}
		\frac{1}{2}\left\| x(t)-x^*\right\| ^2+\frac{1}{2}\left\| y(t)-y^*\right\| ^2\leq\mathcal E(Z_0)e^{-2tC_2(\eta_1,\eta_2,s)\sigma_{\min}^2(A)}.
	\end{equation}
\end{thm}
\begin{proof}
	A standard argument of the well-posedness theory of ordinary differential equations leads to the existence of a unique global $C^1$-smooth solution. Thanks to \cref{thm:con-slp}, 
	\begin{equation}\label{eq:mid-exp}
		\begin{aligned}
			\frac{\rm{d}}{\rm{d}t}\mathcal{E}(Z)={}&\dual{\nabla \mathcal E(Z),Z'}=\dual{\nabla \mathcal{E}(Z) ,\mathcal{G}(Z)} \\
			\leq {}	&-C_1(\eta_1,\eta_2,\theta)\left\langle M(Z),Z-Z^*\right\rangle -C_2(\eta_1,\eta_2,s)\nm{Q(Z-Z^*)}^2.
		\end{aligned}
	\end{equation}
	This implies immediately that
	\[
	\begin{aligned}
		\mathcal E(Z(t))+{}&C_1(\eta_1,\eta_2,\theta)\int_{0}^{t} \left\langle M(Z(r)),Z(r)-Z^*\right\rangle\dd r \\
		{}&\quad +C_2(\eta_1,\eta_2,s)\int_{0}^{t}\nm{Q(Z(r)-Z^*)}^2\dd r\leq \mathcal E(Z_0),
	\end{aligned}
	\]
	which leads to the desired estimate \cref{eq:rate-mod-pdhg-osr}.
	
	Notice the fact 
	\[
	\begin{aligned}
		\dual{M(Z),Z-Z^*} 
		={}&\dual{\nabla F(x)+A^\top y,x-x^*}+\dual{\nabla G(y)-Ax,y-y^*}\\
		\geq {}&F(x)-F(x^*)+\dual{A^\top y,x-x^*}+G(y)-G(y^*)-\dual{Ax,y-y^*}\\
		= {}& \mathcal{L}(x,y^*)-\mathcal{L}(x^*,y).
	\end{aligned}
	\]
	Then by Jensen's inequality, it follows that 
	\[
	\begin{aligned}
		{}&		C_1(\eta_1,\eta_2,\theta)\left[ \mathcal L(\bar{x}(t),y^*)-	 \mathcal L(x^*,\bar{y}(t))\right]\\\leq {}&
		C_1(\eta_1,\eta_2,\theta)\frac{\mathcal L(x_0,y^*)-	 \mathcal L(x^*,y_0)+\int_{0}^{t}\left[\mathcal L(x(r),y^*)-	 \mathcal L(x^*,y(r))\right]\dd r }{1+t}\\
		\leq {}&
		C_1(\eta_1,\eta_2,\theta)\frac{\mathcal L(x_0,y^*)-	 \mathcal L(x^*,y_0)+\int_{0}^{t}\dual{M(Z(r)),Z(r)-Z^*}\dd r }{1+t}\\
		\leq {}&
		\frac{\mathcal E(Z_0) +C_1(\eta_1,\eta_2,\theta)\left[\mathcal L(x_0,y^*)-	 \mathcal L(x^*,y_0)\right]}{1+t}.
	\end{aligned}
	\]
	Similarly, we can prove that 
	\[
	C_2(\eta_1,\eta_2,s)\nm{Q(\bar{Z}(t)-Z^*)}^2\leq  \frac {\mathcal E(Z_0)+C_2(\eta_1,\eta_2,s)\nm{Q(Z_0-Z^*)}^2}{1+t}.
	\]
	Therefore, we finish the proofs of \cref{eq:est1-rate-mod-pdhg-osr,eq:est2-rate-mod-pdhg-osr}.
	
	If $\sigma_{\min}(A)>0$, then by \cref{eq:mid-exp},
	\[
	\begin{aligned}
		\frac{\rm{d}}{\rm{d}t}\mathcal{E}(Z) 
		\leq {}	& -C_2(\eta_1,\eta_2,s)\lambda_{\min}(Q^\top Q)\nm{Z-Z^*}^2=-2C_2(\eta_1,\eta_2,s)\sigma_{\min}^2(A)\mathcal E(Z).
	\end{aligned}
	\]
	This implies \cref{eq:exp-pdhg} immediately and concludes the proof of the theorem.
\end{proof}
\subsection{Correction for discrete-time PDHG}
Let us consider an explicit Euler discretization for the $O(s)$-correction ODE \cref{eq:os-cor-pdhg}:
\begin{equation}\label{eq:ex-osr-pdhg}
	\tag{cPDHG}
	Z_{k+1} = Z_k+s\mathcal G(Z_k)=Z_k-sM(Z_k)+\frac{s^2}{2}\left[\eta_1Q+2\eta_2Q_\theta\right]M(Z_k).
\end{equation}
In component wise, we have 
\[
\left\{
\begin{aligned}
	x_{k+1}={}&x_k-  s\nabla_x\mathcal{L}(x_k,y_k)-\frac{s^2}{2}(\eta_1-2\eta_2 )A^\top\nabla_y\mathcal{L}(x_k,y_k) ,\\
	y_{k+1}={}&y_k+s \nabla_y\mathcal{L}(x_k,y_k)-\frac{s^2}2(\eta_1+2\theta\eta_2)A\nabla_x\mathcal{L}(x_k,y_k). 
\end{aligned}
\right.
\]
Mimicking \cref{eq:lyapunov}, for $Z_k = (x_k, y_k)$,
introduce the discrete Lyapunov function
\[
\mathcal{E}(Z_k)=\frac{1}{2}\left\| x_k-x^*\right\| ^2+\frac{1}{2}\left\| y_k-y^*\right\| ^2.
\]

Using the strong Lyapunov property in \cref{thm:con-slp}, we give the rate of convergence.
\begin{thm}\label{thm:CP-ODE-proof}
	Suppose $F\in\mathcal F_{L_{f}}^1(\R^n)$ and $G\in \mathcal{F}_{L_g}^{1}(\R^m)$. Assume $\theta\geq -1$ and let $\eta_1,\,\eta_2\in\R_+$ be such that $2\eta_2< \eta_1<4-2\theta\eta_2$. If $s$ satisfies 
	\begin{equation}\label{eq:cond-s}\small
		0<s\leq \min\left\{\frac{1}{L},\,\frac{L^2C_1(\eta_1,\eta_2,\theta)}{4L^2+2\nm{Q}^2(\eta^2_1+4\theta^2\eta^2_2)},\,\frac{2L^2(\eta_1-2\eta_2)}{(\eta_1-2\eta_2)L^3+8(2L^2+\nm{Q}^2(\eta^2_1+4\theta^2\eta^2_2))}\right\},
	\end{equation}
	with $L:=\max\{L_f,L_g\}$, then we have 
	\begin{equation}\label{eq:diff-Ek-pdhg-osr}\small
		\mathcal{E}(Z_{k+1})-\mathcal{E}(Z_k)\leq {}
		-\frac{						C_1(\eta_1,\eta_2,\theta)}{2}\left\langle M(Z_k),Z_k-Z^*\right\rangle  -\frac{C_2(\eta_1,\eta_2,s)}{2}\nm{Q(Z_k-Z^*)}^2,
	\end{equation}
	where  both $C_1$ and $C_2$ are defined in \cref{eq:slp-pdhg}. As by products, we get the ergodic rates
	\begin{align}
		\nm{Q(\bar{Z}_k-Z^*)}^2\leq  {}&\frac {2\mathcal E(Z_0)+C_2(\eta_1,\eta_2,s)\nm{Q(Z_0-Z^*)}^2}{C_2(\eta_1,\eta_2,s)(1+k)},
		\label{eq:est1-rate-mod-pdhg-osr-dis}\\
		\mathcal L(\bar{x}_k,y^*)-	 \mathcal L(x^*,\bar{y}_k)\leq{}& \frac{2\mathcal E(Z_0) +C_1(\eta_1,\eta_2,\theta)\left[\mathcal L(x_0,y^*)-	 \mathcal L(x^*,y_0)\right]}{	C_1(\eta_1,\eta_2,\theta)(1+k)},
		\label{eq:est2-rate-mod-pdhg-osr-dis}
	\end{align}
	where 
	\[
	\begin{aligned}
		\bar{x}_k :={}&\frac{1}{1+k}\sum_{i=0}^{k}x_i,\quad 
		\bar{y}_k:={}\frac{1}{1
			+k}\sum_{i=0}^ky_i.
	\end{aligned}
	\]
	Moreover, if $A$ is an invertible square matrix, then $\lambda_{\min}(Q^\top Q)=\sigma_{\min}^2(A)>0$ and we have the linear rate
	\begin{equation}\label{eq:exp-pdhg-dis}
		\frac{1}{2}\left\| x_k-x^*\right\| ^2+\frac{1}{2}\left\| y_k-y^*\right\| ^2\leq\mathcal E(Z_0)\times\left(1-C_2(\eta_1,\eta_2,s)\sigma_{\min}^2(A)\right)^k.
	\end{equation}
\end{thm}
\begin{proof}
	Since $\mathcal{E}(\cdot)$ is quadratic and 1-strongly convex, it is clear that
	\[
	\mathcal{E}(Z_{k+1})-\mathcal{E}(Z_k)=\left\langle \nabla\mathcal{E}(Z_k),Z_{k+1}-Z_k\right\rangle +\frac{1}{2}\left\| Z_{k+1}-Z_k\right\| ^2 .
	\]
	Invoking the explicit scheme \cref{eq:ex-osr-pdhg}  and the strong Lyapunov property (cf. \cref{thm:con-slp}), we obtain 
	\begin{equation}\label{eq:76}
		\begin{aligned}
			{}&	\left\langle \nabla\mathcal{E}(Z_k),Z_{k+1}-Z_k\right\rangle=	\left\langle \nabla\mathcal{E}(Z_k),\mathcal G(Z_k)\right\rangle\\
			\leq {}&
			-C_1(\eta_1,\eta_2,\theta)\left\langle M(Z_k),Z_k-Z^*\right\rangle -C_2(\eta_1,\eta_2,s)\nm{Q(Z_k-Z^*)}^2,
		\end{aligned}
	\end{equation}
	where both $C_1$ and $C_2$ are defined in \cref{eq:slp-pdhg}. By the decomposition $M(Z_k)=H(Z_k)+QZ_k$ and the estimate \cref{eq:H-H}, we claim that
	\[
	\begin{aligned}
		\nm{M(Z_k)-M(Z^*)}^2\leq{}& 2\nm{H(Z_k)-H(Z^*)}^2+2\nm{Q(Z_k-Z^*)}^2\\
		\leq{}&2L\dual{M(Z_k),Z_k-Z^*}+2\nm{Q(Z_k-Z^*)}^2.
	\end{aligned}
	\]
	Hence, it follows from \cref{eq:decomp-G} that 
	\[
	\begin{aligned}
		{}&	 \frac{1}{2}\left\| Z_{k+1}-Z_k\right\| ^2=\frac{s^2}{2}\left\| \mathcal G(Z_k)\right\| ^2
		=\frac{s^2}{2}\left\| \mathcal G(Z_k)-\mathcal G(Z^*)\right\| ^2\\
		={}&
		\frac{s^2}{2}\left\| \left(I-\frac{s}{2}\left[\eta_1Q+2\eta_2Q_\theta\right]\right)\left(M(Z_k)-M(Z^*)\right)\right\| ^2\\
		\leq{}&	 \frac{s^2}{2}\left\| I-\frac{s}{2}\left[\eta_1Q+2\eta_2Q_\theta\right]\right\| ^2
		\left\| M(Z_k)-M(Z^*)\right\| ^2\\
		\leq{}&	s^2\left(2+s^2\nm{Q}^2(\eta^2_1+4\theta^2\eta^2_2)\right)\left(L\dual{M(Z_k),Z_k-Z^*}+\nm{Q(Z_k-Z^*)}^2\right).
	\end{aligned}
	\]
	Thus, we obtain 
	\[
	\begin{aligned}
		\mathcal{E}(Z_{k+1})-\mathcal{E}(Z_k)\leq &
		-\left[
		C_1(\eta_1,\eta_2,\theta)-Ls^2\left(2+s^2\nm{Q}^2(\eta^2_1+4\theta^2\eta^2_2)\right)\right]\left\langle M(Z_k),Z_k-Z^*\right\rangle \\
		&\quad -\left[C_2(\eta_1,\eta_2,s)-s^2\left(2+s^2\nm{Q}^2(\eta^2_1+4\theta^2\eta^2_2)\right)\right]\nm{Q(Z_k-Z^*)}^2.
	\end{aligned}
	\]
	By \cref{eq:cond-s} , we find that 
	\[
	\begin{aligned}
		Ls^2\left(2+s^2\nm{Q}^2(\eta^2_1+4\theta^2\eta^2_2)\right)\leq{}& \frac{s}{L^2}\left(2L^2+\nm{Q}^2(\eta^2_1+4\theta^2\eta^2_2)\right)\leq \frac{			C_1(\eta_1,\eta_2,\theta)}{2},
	\end{aligned} 
	\]
	and similarly, we have
	\[
	s^2\left(2+s^2\nm{Q}^2(\eta^2_1+4\theta^2\eta^2_2)\right)\leq  \frac{			C_2(\eta_1,\eta_2,\theta)}{2}.
	\]
	Consequently, we obtain the contraction estimate \cref{eq:diff-Ek-pdhg-osr}.
	
	Following \cref{eq:est1-rate-mod-pdhg-osr,eq:est2-rate-mod-pdhg-osr}, it is not hard to establish \cref{eq:est1-rate-mod-pdhg-osr-dis,eq:est2-rate-mod-pdhg-osr-dis}. Moreover, if $\sigma_{\min}(A)>0$, then from \cref{eq:diff-Ek-pdhg-osr}, 
	\[\small
	\mathcal{E}(Z_{k+1})-\mathcal{E}(Z_k)\leq {}  -\frac{C_2(\eta_1,\eta_2,s)}{2}\lambda_{\min}(Q^\top Q)\nm{Z_k-Z^*}^2=-C_2(\eta_1,\eta_2,s)\sigma^2_{\min}(A)\mathcal E(Z_k),
	\]
	which yields the linear rate \cref{eq:exp-pdhg-dis} and  completes the proof of the theorem.
\end{proof}

\section{HB with $O(\sqrt{s})$-Correction}\label{sec:ode to dta-hb}
In this section, we extend the high-resolution term correction idea to the well-known \cref{eq:hb} method. As mentioned in \cref{sec:intro-low}, with careful choice of parameters, \cref{eq:hb} converges with provable suboptimal rate; however, for well-chosen parameters that are optimal for quadratic problems, \cref{eq:hb} might diverge for general  smooth strongly convex objectives. Observing \cref{eq:osr-hb-polyak,eq:osr-nag-mu}, the $O(1)$-resolution ODEs of \cref{eq:hb,eq:nag} are the same. Therefore, the reason why \cref{eq:hb} diverges while \cref{eq:nag} converges with optimal rate lies in the subtle $O(\sqrt{s})$-term. This naturally suggests an $O(\sqrt{s})$-correction for both the continuous and discrete \cref{eq:hb} based on the high-resolution ODE \cref{eq:osr-nag-mu} of \cref{eq:nag}. 
\subsection{Correction for continuous-time HB}
Again, to avoid using the Hessian information, we replace $\nabla^2 F(x)$ with  $\mu I$ and consider the following $O(\sqrt{s})$-correction ODE:
\begin{equation}\label{eq:os-correc-hb}
	\begin{bmatrix}
		x \\ v
	\end{bmatrix}' =  
	\begin{bmatrix}
		v \\ -2\sqrt{\mu} v - \nabla F(x)
	\end{bmatrix}  + 
	\frac{	\eta \sqrt{s} }{2} \begin{bmatrix}
		-6\sqrt{\mu}v - 3\nabla F(x)  \\
		\mu v - 2\sqrt{\mu} \nabla F(x)
	\end{bmatrix},
\end{equation}
where $\eta>0$ is a weight parameter. Letting $w:=v/\sqrt{\mu}+x$ yields an equivalent presentation 
\begin{equation}\label{eq:chb}\small
	\begin{bmatrix}
		x \\ w
	\end{bmatrix}' = \mathcal G(x,w),\quad \mathcal G(x,w):=
	\begin{bmatrix}
		\sqrt{\mu}(1-3\eta \sqrt{\mu s})(w-x)- \frac{3}{2}\eta\sqrt{s}\nabla F(x)\\
		\frac{\sqrt{\mu}}{2}(2+5\eta\sqrt{\mu s})(x-w)-\frac{1}{2\sqrt{\mu }}(2+ 5\eta\sqrt{\mu s})\nabla F(x)
	\end{bmatrix}.
\end{equation} 

In what follows, we present a Lyapunov analysis of the continuous model \cref{eq:chb}. The  Lyapunov function is given as below
\begin{equation}\label{eq:hbLyapunov}
	\mathcal{E} (x,w):= F(x)-F(x^*) +\frac{b}{2}\left\| w-x^*\right\| ^2,
\end{equation}
where $b:=\mu (1-3\eta\sqrt{\mu s})/(1+5\eta \sqrt{\mu s}/2)$.
The exponential decay is established via the strong Lyapunov property.
\begin{thm}
	Assume $\eta >0$ and $F\in \mathcal{S}_{\mu,L}^{1}(\R^n)$ with $L\geq\mu>0$. There exists a unique global $C^1$-smooth solution $ (x(t),w(t))$ to the $O(\sqrt{s})$-correction HB ODE \cref{eq:chb} with $(x(0),w(0))=(x_0,w_0)\in\R^n\times\R^n$. Moreover, if $3\eta\sqrt{\mu s}<1$, then we have the strong Lyapunov property
	\begin{equation}\label{eq:chb-slp}
		-	\dual{\nabla \mathcal{E}(x,w), \mathcal{G}(x,w)}\geq \sqrt{\mu }(1-3\eta \sqrt{\mu s})\mathcal{E}(x,w)+\frac{3 \eta \sqrt{s}}{2}\left\| \nabla F(x)\right\|^2,
	\end{equation}
	for all $(x,w)\in\R^n\times\R^n$. This implies that 
	\begin{equation}\label{eq:exp-chb-slp}
		\mathcal{E} (x(t),w(t))\leq e^{-\sqrt{\mu }(1-3\eta \sqrt{\mu s})t}\mathcal{E} (x_0,w_0),\quad \forall\,t>0.
	\end{equation}
\end{thm}
\begin{proof}
	It is easy to show the exists and uniqueness of the  global $C^1$-smooth solution $ (x(t),w(t))$. Let us verify the strong Lyapunov property \cref{eq:chb-slp}. Observing \cref{eq:chb,eq:hbLyapunov}, a direct computation leads to
	\[
	\begin{aligned}
		{}& -\dual{\nabla \mathcal{E}(x,w), \mathcal{G}(x,w)} \\
		={}& - \sqrt{\mu }(1-3\eta \sqrt{\mu s})\left\langle \nabla F(x),w-x\right\rangle +\frac{3}{2}\eta \sqrt{s}\left\| \nabla F(x)\right\| ^2 \\
		{}&\quad+\frac{b\sqrt{\mu}}{2}(2+5\eta \sqrt{\mu s})\left\langle w-x^*,w-x \right\rangle+\frac{b}{2\sqrt{\mu}}(2+5\eta \sqrt{\mu s})\left\langle \nabla F(x),w-x^*\right\rangle \\
		={}	&  \sqrt{\mu }(1-3\eta\sqrt{\mu s})\left\langle \nabla F(x),x-x^*\right\rangle+\frac{3}{2}\eta \sqrt{s}\left\| \nabla F(x)\right\|^2  \\ 
		{}&\quad -\frac{b\sqrt{\mu}}{2}(2+5\eta \sqrt{\mu s})\left\langle x-w,w-x^* \right\rangle.
	\end{aligned}
	\]
	Due to the strongly convex property, the first cross term is bounded above by 
	\[\small
	\begin{aligned}
		\left\langle \nabla F(x),x-x^*\right\rangle
		\geq     F(x)-F(x^*)+\frac{\mu}{2}\|x-x^*\|^2,
	\end{aligned}
	\]
	and the last cross term can be expanded to 
	\[
	\left\langle x-w,w-x^* \right\rangle=\frac{1}{2} \left(\left\| x-x^*\right\|^2-\left\| x-w\right\| ^2-\left\| w-x^*\right\|^2\right).
	\] 
	Adding all together, we get
	\[
	\begin{aligned}
		-\dual{\nabla \mathcal{E}(x,w), \mathcal{G}(x,w)} 
		\geq{}&   \sqrt{\mu }(1-3\eta\sqrt{\mu s})\left( F(x)-F(x^*) \right)+\frac{3 }{2}\eta \sqrt{s}\left\| \nabla F(x)\right\|^2\\
		{}&\quad +\frac{b\sqrt{\mu}}{4}(2+5\eta \sqrt{\mu s})\left(\left\| x-w\right\| ^2+\left\| w-x^*\right\|^2\right)\\
		={}&  \sqrt{\mu }(1-3\eta\sqrt{\mu s})\mathcal{E}(x,w)+\frac{3 }{2}\eta \sqrt{s}\left\| \nabla F(x)\right\|^2 \\
		{}&\quad +\frac{b\sqrt{\mu}}{4}(2+5\eta \sqrt{\mu s})\left\| x-w\right\| ^2+\frac{11\eta b}{4}\mu\sqrt{s}\left\| w-x^*\right\|^2 \\
		\geq{}	&  \sqrt{\mu }(1-3\eta\sqrt{\mu s})\mathcal{E}(x,w)+\frac{3 }{2}\eta \sqrt{s}\left\| \nabla F(x)\right\|^2. 
	\end{aligned}
	\]
	This implies the strong Lyapunov property \cref{eq:chb-slp}.
	
	Let $ (x(t),w(t))$ be the global $C^1$-smooth solution. Notice that 
	\[
	\frac{\dd}{\dd t}\mathcal E(x,w)=  \dual{\nabla \mathcal{E}(x(t),w(t)), \mathcal{G}(x(t),w(t))}\leq -\sqrt{\mu }(1-3\eta\sqrt{\mu s})\mathcal{E}(x(t),w(t)),
	\]
	which yields immediately the exponential rate \cref{eq:exp-chb-slp} and thus  completes the proof. 
\end{proof}

	\subsection{Correction for discrete-time HB}
	Now, let us consider a semi-implicit scheme for the $O(\sqrt{s})$-correction ODE \cref{eq:chb}: 
	\begin{equation}\label{eq:chb-dis}\small
		\tag{cHB}
		\left\{
		\begin{aligned}\
			\frac{x_{k+1}-x_k}{\sqrt{s}} ={}&\sqrt{\mu}(1-3\eta \sqrt{\mu s})(w_k-x_{k+1})-\frac{3}{2}\eta \sqrt{s}\nabla F(x_k),\\	
			\frac{w_{k+1}-w_k}{\sqrt{s}}={}&-  \frac{\sqrt{\mu}}{2}(2+5\eta\sqrt{\mu s})(w_{k+1}-x_{k+1})- \frac{1}{2\sqrt{\mu }}(2+ 5\eta\sqrt{\mu s})\nabla F(x_{k+1}),
		\end{aligned}
		\right.
	\end{equation}
	which leads to an $O(\sqrt{s})$-correction to the original \cref{eq:hb} method. Based on the discrete analogue to \cref{eq:hbLyapunov}:
	\begin{equation}\label{eq:lk}
		\mathcal{E}_k:=\mathcal{E}(x_k,w_k)= F(x_k)-F(x^*) +\frac{b}{2}\|w_k-x^*\|^2, 
	\end{equation}
	and the strong Lyapunov property \cref{eq:chb-slp}, we are able to establish the optimal linear convergence rate of the correction scheme \cref{eq:chb-dis}. 
	\begin{thm}
		Assume $F\in \mathcal{S}_{\mu,L}^{1}(\R^n)$ with $L\geq\mu>0$.  Let $\{(x_k,w_k)\}$ be generated by \cref{eq:chb-dis} with $			3\eta\sqrt{\mu s}>1$ and
		\begin{equation}\label{eq:cond-s-chb}
			9 L\eta^2s+4(1+5\eta\sqrt{\mu s}/2)(1-3\eta \sqrt{\mu s})   \leq 12 \eta  ,
		\end{equation}
		then we have 
		\begin{equation}\label{eq:dis-chb-slp}
			\mathcal E_{k+1}-\mathcal E_k\leq -\sqrt{\mu s}(1 - 3\eta \sqrt{\mu s})\mathcal{E}_{k+1},
		\end{equation}
		which implies that 
		\begin{equation}\label{eq:exp-dis-chb-slp}
			\mathcal{E}_k\leq\left(1+\sqrt{\mu s}(1-3\eta\sqrt{\mu s})\right)^{-k}\mathcal{E} _0,\quad \forall\,k\in\mathbb N.
		\end{equation} 
	\end{thm}
	
	\begin{proof}
		The linear rate \cref{eq:exp-dis-chb-slp} follows from the contraction estimate \cref{eq:dis-chb-slp} easily. Thus, it is sufficient to establish \cref{eq:dis-chb-slp}.
		In view of \cref{eq:chb,eq:chb-dis}, we have
		\[
		\left\{
		\begin{aligned}
			w_{k+1}={}&w_k+\sqrt{s}\mathcal G_w(x_{k+1},w_{k+1}),\\
			x_{k+1}={}&x_k+\sqrt{s}\mathcal G_x(x_{k+1},w_{k+1})+\sqrt{s}\Delta_{k+1},
		\end{aligned}
		\right.
		\]
		where $\mathcal G_x$ and $\mathcal G_w$ are respectively the first and the second component of $\mathcal G$ defined in \cref{eq:chb} and 
		\[
		\Delta_{k+1}:=\sqrt{\mu}(1-3\eta \sqrt{\mu s})(w_k-w_{k+1})+\frac{3}{2}\eta \sqrt{s}\left(\nabla F(x_{k+1})-\nabla F(x_k)\right).
		\]	
		Following the proof of \cref{thm:CP-ODE-proof}, we start from the difference
		\[
		\begin{aligned}
			{}&			\mathcal{E}_{k+1}-\mathcal{E}_{k}=  \mathcal{E}\left(x_{k+1}, w_{k}\right)-\mathcal{E}\left(x_{k}, w_{k}\right)  
			+\mathcal{E}\left(x_{k+1}, w_{k+1}\right)-\mathcal{E}\left(x_{k+1}, w_{k}\right)\\
			={}&F(x_{k+1})-F(x_k) +\left\langle \nabla_w \mathcal{E}(x_{k+1},w_{k+1}),w_{k+1}-w_k\right\rangle -\frac{b}{2}\left\| w_{k+1}-w_k\right\| ^2\\
			={}&F(x_{k+1})-F(x_k) +\sqrt{s}\left\langle \nabla_w\mathcal{E}(x_{k+1},w_{k+1}),\mathcal{G}_w(x_{k+1},w_{k+1})\right\rangle -\frac{b}{2}\left\| w_{k+1}-w_k\right\| ^2.
		\end{aligned}
		\] 
		We now estimate the first term as follows
		\[
		\begin{aligned}
			F(x_{k+1})-F(x_k)\leq 	{}		& \left\langle \nabla F(x_{k+1}),x_{k+1}-x_k\right\rangle -\frac{1}{2L}\left\| \nabla F(x_{k+1})-\nabla F(x_k)\right\| ^2\\
			={}	&\left\langle \nabla_x\mathcal{E}(x_{k+1},w_{k+1}),x_{k+1}-x_k\right\rangle -\frac{1}{2L}\left\| \nabla F(x_{k+1})-\nabla F(x_k)\right\| ^2\\
			={}	&\sqrt{s}\left\langle \nabla_x\mathcal{E}(x_{k+1},w_{k+1}),\mathcal{G}_x(x_{k+1},w_{k+1})\right\rangle -\frac{1}{2L}\left\| \nabla F(x_{k+1})-\nabla F(x_k)\right\| ^2\\
			{}&\quad + \sqrt{s}\dual{\nabla F(x_{k+1}),\Delta_{k+1}}.
		\end{aligned}
		\]
		This implies that 
		\[
		\begin{aligned}
			\mathcal{E}_{k+1}-\mathcal{E}_{k} 
			\leq {} &\sqrt{s}\left\langle \nabla\mathcal{E}(x_{k+1},w_{k+1}),\mathcal{G}(x_{k+1},w_{k+1})\right\rangle -\frac{b}{2}\left\| w_{k+1}-w_k\right\| ^2\\
			{}&\quad	-\frac{1}{2L}\left\| \nabla F(x_{k+1})-\nabla F(x_k)\right\| ^2
			+ \sqrt{s}\dual{\nabla F(x_{k+1}),\Delta_{k+1}}.
		\end{aligned}
		\]
		Since $3\eta\sqrt{\mu s}<1$, invoking the strong Lyapunov property \cref{eq:chb-slp}, we obtain 
		\[
		\begin{aligned}
			\mathcal{E}_{k+1}-\mathcal{E}_{k} 
			\leq {} &-\sqrt{\mu s}(1 - 3\eta \sqrt{\mu s})\mathcal{E}_{k+1}- \frac{3\eta s}{2}  \|\nabla F(x_{k+1})\|^2-\frac{b}{2}\left\| w_{k+1}-w_k\right\| ^2\\
			{}&\quad	-\frac{1}{2L}\left\| \nabla F(x_{k+1})-\nabla F(x_k)\right\| ^2
			+ \sqrt{s}\dual{\nabla F(x_{k+1}),\Delta_{k+1}}.
		\end{aligned}
		\]
		Let us focus on the last cross term:
		\[
		\begin{aligned}
			\dual{\nabla F(x_{k+1}),\Delta_{k+1}}={}&\sqrt{\mu }(1-3\eta \sqrt{\mu s}) \left\langle  \nabla F(x_{k+1}), w_k-w_{k+1} \right\rangle\\
			{}&\quad  +\frac{3\eta \sqrt{s}}{2}\left\langle  \nabla F(x_{k+1}),\nabla F(x_{k+1})-\nabla F(x_k)\right\rangle\\
			\leq {}&\frac{1}{2L\sqrt{s}}\left\| \nabla F(x_{k+1})-\nabla F(x_k)\right\| ^2+ \frac{9L\eta^2s^{3/2}}{8}\left\| \nabla F(x_{k+1})\right\| ^2 \\
			{}&\quad +\frac{b}{2\sqrt{s}}\left\| w_{k+1}-w_k\right\|^2+\frac{ \mu\sqrt{s} }{2b}(1-3\eta \sqrt{\mu s})^2 \left\| \nabla F(x_{k+1})\right\|^2,
		\end{aligned}
		\]
		with $b = \mu(1-3\eta\sqrt{\mu s})/(1+5\eta\sqrt{\mu s}/2)$.
		Consequently, it follows that
		\[
		\begin{aligned}
			\mathcal{E}_{k+1}-\mathcal{E}_{k} 
			\leq {} &-\sqrt{\mu s}(1 - 3\eta \sqrt{\mu s})\mathcal{E}_{k+1}+\frac{a(s)s}{8 } \|\nabla F(x_{k+1})\|^2,
		\end{aligned}
		\]
		where $a(s)=9 L\eta^2s+4(1+5\eta\sqrt{\mu s}/2)(1-3\eta \sqrt{\mu s})   - 12 \eta  \leq0 $ by  \cref{eq:cond-s-chb}. This implies the contraction estimate \cref{eq:dis-chb-slp} immediately and concludes the proof of this theorem.
	\end{proof}
	\begin{rem}
		The final rate given by \cref{eq:exp-dis-chb-slp} is $(1+\rho(\eta,s))^{-k}$ with $\rho(\eta,s):= \sqrt{\mu s}(1-3\eta\sqrt{\mu s})=O(1)\sqrt{\mu/L}$, which, ignoring the constant $O(1)$, is optimal with respect to the condition number $\sqrt{L/\mu}$. In particular, the optimal choice 
		\begin{equation}\label{eq:opt-eta-s}
			\eta^{*}=\dfrac{\sqrt{L}(11\sqrt{\mu}+6\sqrt{L})}{9(2\sqrt{\mu}+\sqrt{L})^{2}},\quad  s^*=\frac{36(2\sqrt{\mu}+\sqrt{L})^2}{L(11\sqrt{\mu}+6\sqrt{L})^2},
		\end{equation}
		yields the maximal contraction constant
		\[	
		\rho(\eta^*,s^*) = 	\dfrac{6\sqrt{\mu}}{11\sqrt{\mu}+6\sqrt{L}}.
		\] 
	\end{rem}

\section{Numerical Experiments}
\label{sec:ex-hb}
In this section, we present two simple examples to show the performances of our $O(\sqrt{s})$-correction schemes \cref{eq:ex-osr-pdhg,eq:chb-dis}.
\subsection{A high-dimensional counterexample for PDHG}
He et al. \cite[Section 3]{he_convergence_2022} showed that PDHG \cref{eq:pdhg} (also known as the Arrow-Hurwicz method) diverges for the bilinear saddle point problem：
\begin{equation}\label{eq:bspp}
	\min_{x \in \mathbb{R}^n}\;\max_{y \in \mathbb{R}^m} {\mathcal{L}(x,y)=y^TAx},
\end{equation} 
with $A\in \mathbb{R}^{m\times n}(n\geq m)$. In \cref{fig:bsp} we report the numerical results of  PDHG, CP and \cref{eq:ex-osr-pdhg} with the step size $s=0.5/\nm{A}$. For \cref{eq:ex-osr-pdhg}, we take $\eta_1=3/2,\eta_2=1/12$ and $\theta=1$.
%
As shown in \cref{fig:bsp}, although its convergence rate is slightly slower than that of CP, the modified method \cref{eq:ex-osr-pdhg} achieves significantly improved performance than  PDHG, which exhibits persistent limit cycle behavior and fails to converge. 
\begin{figure}[H]
	\centering  
	\subfloat[$m=n=1$.]
	{
		\begin{minipage}[t]{0.3\textwidth}
			\centering          
			\includegraphics[width=0.9\textwidth]{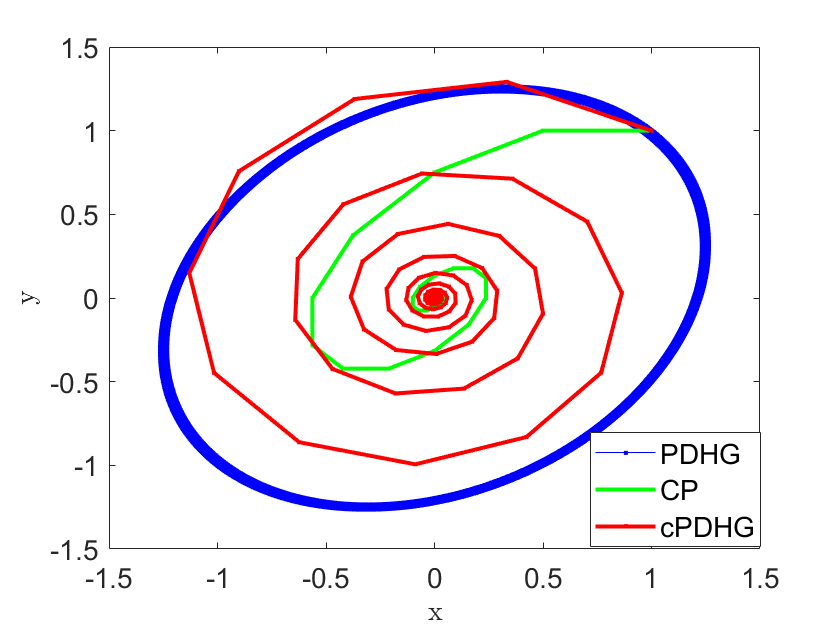}   
		\end{minipage}%
	}\centering  
	\subfloat[$m=5,n=8$.] 
	{
		\begin{minipage}[t]{0.3\textwidth}
			\centering          
			\includegraphics[width=0.9\textwidth]{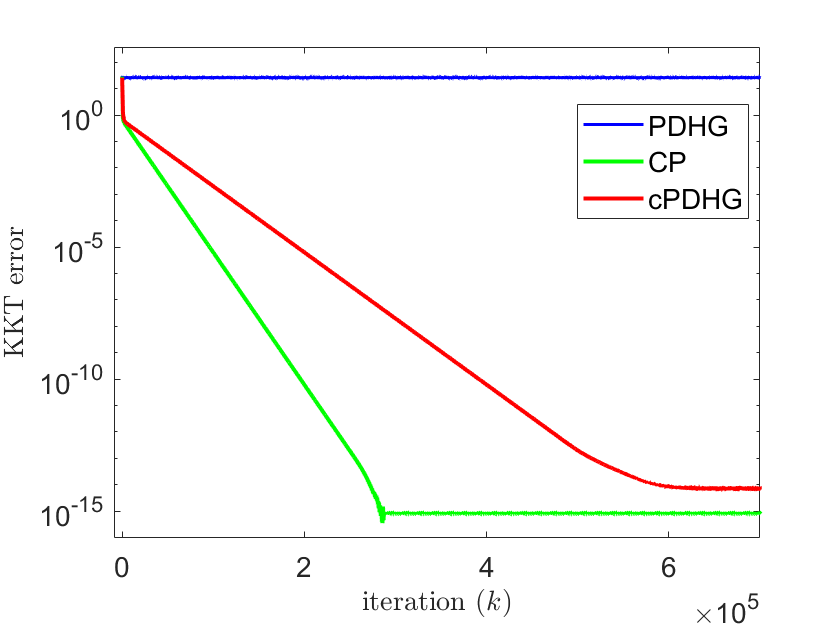}   
		\end{minipage}
	}
	\subfloat[$m=n=50$.]
	{
		\begin{minipage}[t]{0.3\textwidth}
			\centering          
			\includegraphics[width=0.9\textwidth]{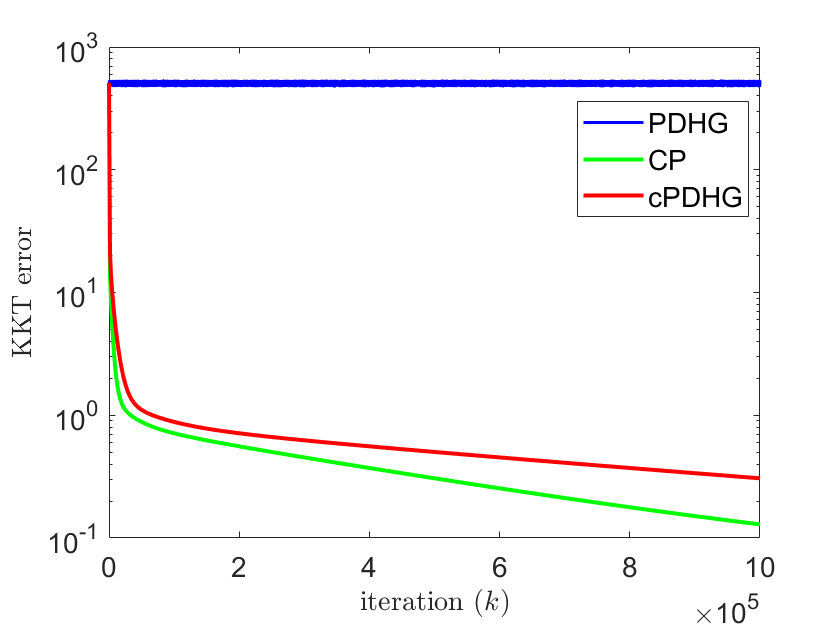}   
		\end{minipage}%
	}
	\caption{Numerical results for the bilinear saddle point problems \cref{eq:bspp}.} 
	\label{fig:bsp}
\end{figure} 
\subsection{An illustrative example for the divergence of HB}
\label{sec:ex-hb1}
A well-known one-dimensional counterexample is given by \cite{lessard_analysis_2016}, where the objective $F$ has the following piecewise linear gradient:
\begin{equation}\label{eq:hb-counterexample}\small
	\nabla F(x) =
	\begin{cases}
		25x, & \text{if } x < 1, \\
		x + 24, & \text{if } 1 \leq x < 2, \\
		25x - 24, & \text{if } x \geq 2.
	\end{cases}
\end{equation}
This gradient function is continuous and monotone, and the primal function $F$ belongs to the class $\mathcal{S}^1_{\mu,L}$ with $\mu = 1$ and $L = 25$. As reported in \cite{lessard_analysis_2016}, if we take Polyak's parameter setting (cf.\cref{eq:s-beta-hb}): $\beta = (1 - \sqrt{\mu s})^2 = 4/9$ and $s = 4/{(\sqrt{L} + \sqrt{\mu})^2} = 1/9$, then \cref{eq:hb} exhibits pathological oscillatory behavior for the initial value $3.07 \leq x_0 \leq 3.46$; see \cref{fig:HB-noncovergence}.  
\begin{figure}[H]
	\centering  
	\subfloat[$s=1/9,x_1=x_0=3.25$]
	{
		\begin{minipage}[t]{0.35\textwidth}
			\centering          
			\includegraphics[width=0.8\textwidth]{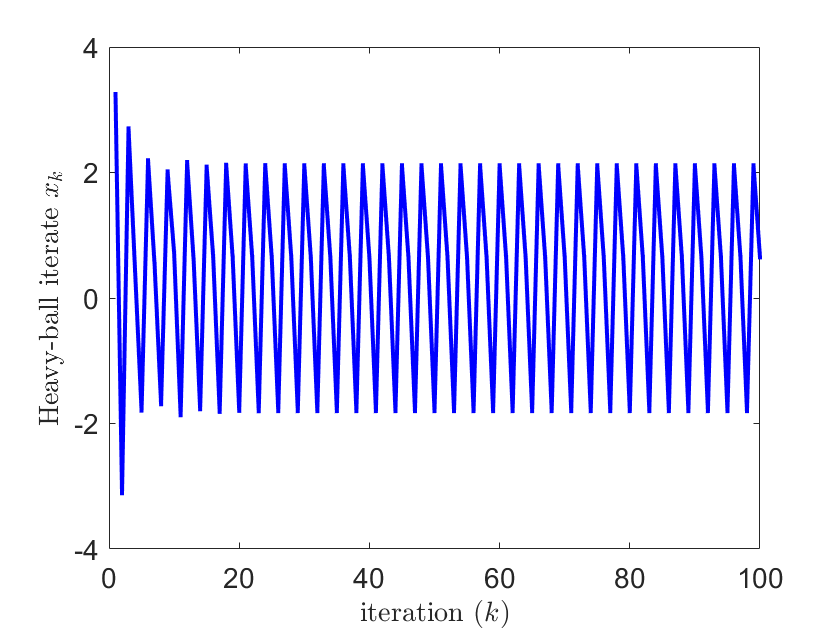}   
		\end{minipage}%
		\label{fig:HB-noncovergence}
	}\centering  
	\subfloat[$\eta=0.465,s=0.042,(x_0,w_0)=(3.25,0)$]
	{
		\begin{minipage}[t]{0.35\textwidth}
			\centering          
			\includegraphics[width=0.8\textwidth]{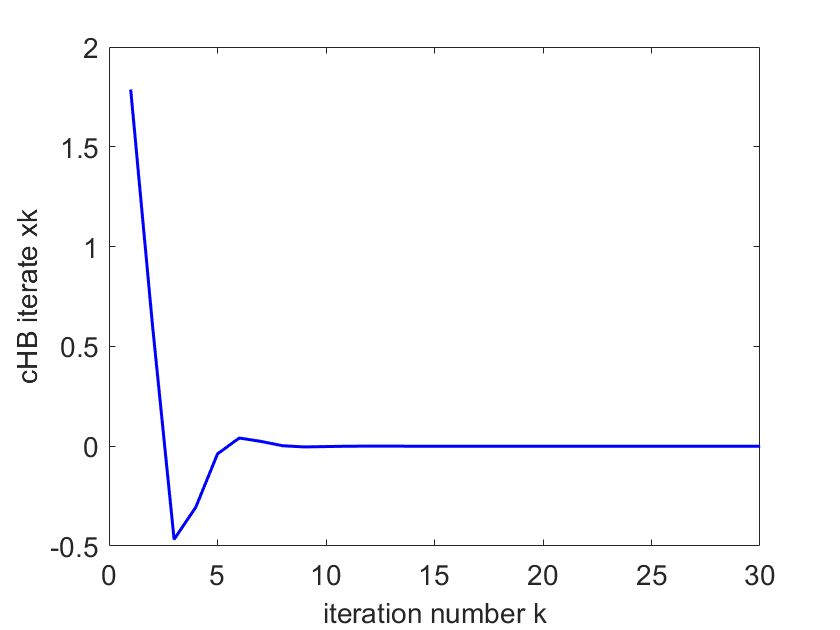}   
		\end{minipage}%
		\label{fig:hb-covergence} 
	}
	
	\centering
	\subfloat[$s=0.042,x_1=x_0=3.25$]
	{
		\begin{minipage}[t]{0.35\textwidth}
			\centering          
			\includegraphics[width=0.8\textwidth]{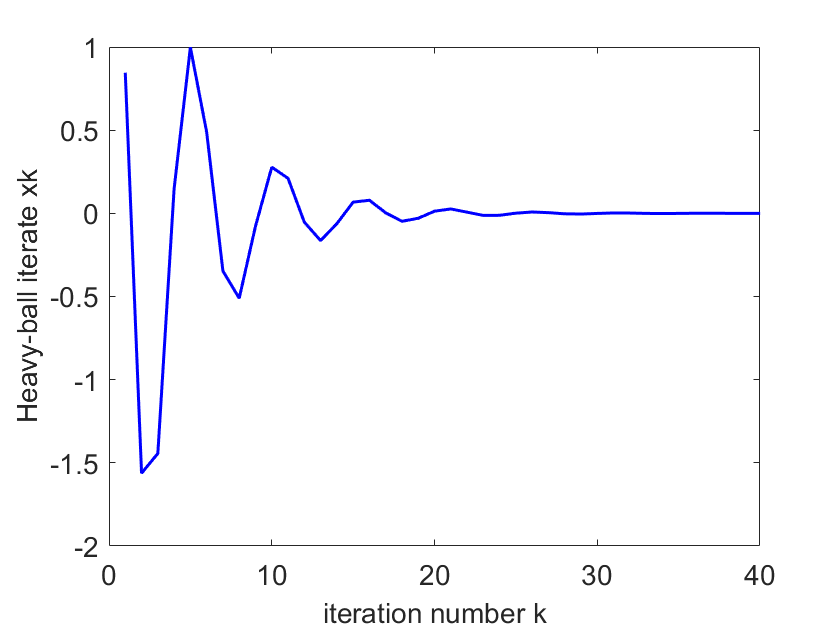}   
		\end{minipage}%
		\label{fig:HB-noncovergence1}
	}\centering  
	\subfloat[$\eta=0.465,s=0.042$]
	{
		\begin{minipage}[t]{0.35\textwidth}
			\centering          
			\includegraphics[width=0.8\textwidth]{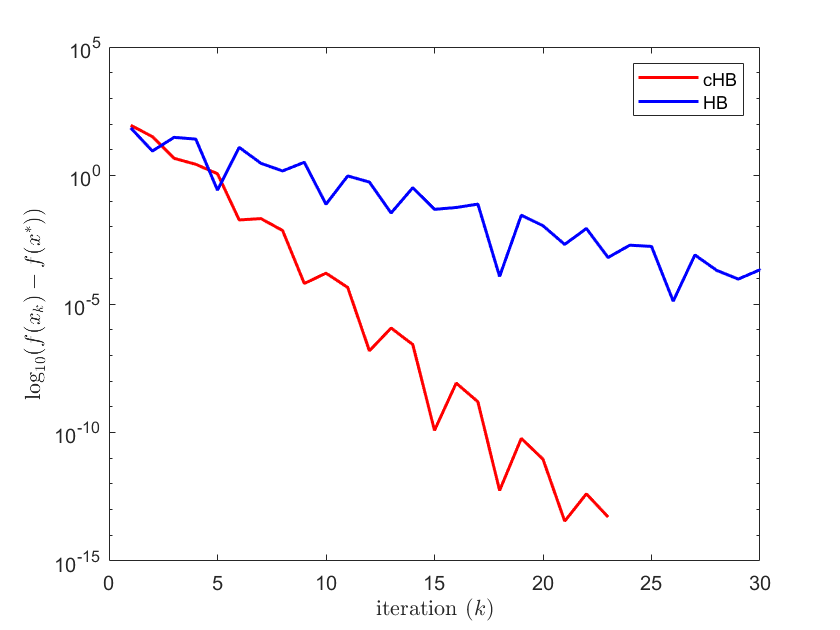}   
		\end{minipage}%
		\label{fig:hb-covergence1} 
	}
	\caption{Performances of different algorithms for minimizing the objective defined by \cref{eq:hb-counterexample}}
	\label{fig:hb}
\end{figure}

For comparison,  we also examine the behavior of \cref{eq:chb-dis} on this counterexample with the optimal choice \cref{eq:opt-eta-s}: $\eta^*=205/441\approx 0.465$ and $s^*=
1764/42025\approx 0.042$. For this step size, both \cref{eq:hb,eq:chb-dis} converge to the optimal solution $x^*=0$; see \cref{fig:hb-covergence,fig:HB-noncovergence1}.
However, \cref{eq:chb-dis} converges more stably and exhibits faster rate, as shown in \cref{fig:hb-covergence1}. 

\section{Conclusion}
\label{sec:con}
In this work, we propose a unified high-resolution ODE framework for the analysis of accelerated gradient methods with momentum and variable parameters. Our work extends the $O(s^r)$-resolution framework by Lu \cite{lu_osr-resolution_2022} without the fixed-point assumption. Also, a careful investigation on NAG and HB rebuilds the result by Shi et al. \cite{shi_understanding_2021}: the hidden gradient correction or Hessian-driven damping makes NAG more stable than HB, which only involves the velocity correction. In addition, we propose a high-order correction approach for HB and PDHG, and prove the optimal convergence rates via the Lyapunov analysis.
	
	\bibliographystyle{abbrv}

\begin{thebibliography}{10}
	
	\bibitem{Apidopoulos2025}
	V.~Apidopoulos, C.~Molinari, J.~Peypouquet, and S.~Villa.
	\newblock Preconditioned primal-dual dynamics in convex optimization:
	non-ergodic convergence rates.
	\newblock {\em arXiv:2506.00501v1}, 2025.
	
	\bibitem{Attouch2020f}
	H.~Attouch, Z.~Chbani, J.~Fadili, and H.~Riahi.
	\newblock First-order optimization algorithms via inertial systems with
	{Hessian} driven damping.
	\newblock {\em Math. Program.}, 193:113--155, 2020.
	
	\bibitem{attouch_fast_2018}
	H.~Attouch, Z.~Chbani, J.~Peypouquet, and P.~Redont.
	\newblock Fast convergence of inertial dynamics and algorithms with asymptotic
	vanishing viscosity.
	\newblock {\em Math. Program.}, 168(1):123--175, 2018.
	
	\bibitem{attouch_rate_2019}
	H.~Attouch, Z.~Chbani, and H.~Riahi.
	\newblock Rate of convergence of the {Nesterov} accelerated gradient method in
	the subcritical case $b\leqslant 3$.
	\newblock {\em ESAIM Control Optim. Calc. Var.}, 25(2), 2019.
	
	\bibitem{Barrett2021}
	D.~G.~T. Barrett and B.~Dherin.
	\newblock Implicit gradient regularization.
	\newblock In {\em 9th International Conference on Learning Representations,
		ICLR.}, 2021.
	
	\bibitem{chambolle_first-order_2011}
	A.~Chambolle and T.~Pock.
	\newblock A first-order primal-dual algorithm for convex problems with
	applications to imaging.
	\newblock {\em J. Math. Imaging Vision}, 40(1):120--145, 2011.
	
	\bibitem{chen_first_2019}
	L.~Chen and H.~Luo.
	\newblock First order optimization methods based on {Hessian-driven Nesterov}
	accelerated gradient flow.
	\newblock {\em arXiv:1912.09276}, 2019.
	
	\bibitem{chen_luo_unified_2021}
	L.~Chen and H.~Luo.
	\newblock A unified convergence analysis of first order convex optimization
	methods via strong {Lyapunov} functions.
	\newblock {\em arXiv:2108.00132}, 2021.
	
	\bibitem{esser_general_2010}
	E.~Esser, X.~Zhang, and T.~F. Chan.
	\newblock A general framework for a class of first order primal-dual algorithms
	for convex optimization in imaging science.
	\newblock {\em SIAM J. Imaging Sci.}, 3(4):1015--1046, 2010.
	
	\bibitem{fu_2023_understanding}
	P.~Fu and Z.~Tan.
	\newblock Understanding accelerated gradient methods: Lyapunov analyses and
	{Hamiltonian} assisted interpretations.
	\newblock {\em arxiv:2304.10063}, 2023.
	
	\bibitem{ghadimi_global_2015}
	E.~Ghadimi, H.~R. Feyzmahdavian, and M.~Johansson.
	\newblock Global convergence of the {Heavy}-ball method for convex
	optimization.
	\newblock In {\em 2015 {European} {Control} {Conference} ({ECC})}, pages
	310--315, Linz, Austria, 2015. IEEE.
	
	\bibitem{goujaud2025provable}
	B.~Goujaud, A.~Taylor, and A.~Dieuleveut.
	\newblock Provable non-accelerations of the heavy-ball method.
	\newblock {\em Math. Program.}, pages 1--59, 2025.
	
	\bibitem{Hairer_2006_Geometric}
	E.~Hairer, C.~Lubich, and G.~Wanner.
	\newblock {\em Geometric Numerical Integration: Structure-Preserving Algorithms
		for Ordinary Differential Equations}, volume~31 of {\em Springer Series in
		Computational Mathematics}.
	\newblock Springer-Verlag, Berlin, 2nd edition, 2006.
	
	\bibitem{he_convergence_2022}
	B.~He, S.~Xu, and X.~Yuan.
	\newblock On convergence of the {Arrow}–{Hurwicz} method for saddle point
	problems.
	\newblock {\em J. Math. Imaging Vision}, 64(6):662--671, 2022.
	
	\bibitem{he_convergence_2014}
	B.~He, Y.~You, and X.~Yuan.
	\newblock On the {convergence} of {primal}-{dual} {hybrid} {gradient}
	{algorithm}.
	\newblock {\em SIAM J. Imaging Sci.}, 7(4):2526--2537, 2014.
	
	\bibitem{krichene_accelerated_2015}
	W.~Krichene, A.~Bayen, and P.~L. Bartlett.
	\newblock Accelerated mirror descent in continuous and discrete time.
	\newblock {\em Advances in Neural Information Processing Systems (NIPS)}, 28,
	2015.
	
	\bibitem{lessard_analysis_2016}
	L.~Lessard, B.~Recht, and A.~Packard.
	\newblock Analysis and {design} of {optimization} {algorithms} via {integral}
	{quadratic} {constraints}.
	\newblock {\em SIAM J. Optim.}, 26(1):57--95, 2016.
	
	\bibitem{Li2024}
	B.~Li and B.~Shi.
	\newblock Understanding the {ADMM} algorithm via high-resolution differential
	equations.
	\newblock {\em arXiv:2401.07096}, 2024.
	
	\bibitem{Li2024a}
	B.~Li and B.~Shi.
	\newblock Understanding the {PDHG} algorithm via high-resolution differential
	equations.
	\newblock {\em arXiv:2403.11139v1}, 2024.
	
	\bibitem{lu_osr-resolution_2022}
	H.~Lu.
	\newblock An ${O}(s^r)$-resolution {{ODE}} framework for understanding
	discrete-time algorithms and applications to the linear convergence of
	minimax problems.
	\newblock {\em Math. Program.}, 194:1061--1112, 2022.
	
	\bibitem{luo_acc_primal-dual_2021}
	H.~Luo.
	\newblock Accelerated primal-dual methods for linearly constrained convex
	optimization problems.
	\newblock {\em arXiv:2109.12604}, 2021.
	
	\bibitem{luo_primal-dual_2022}
	H.~Luo.
	\newblock A primal-dual flow for affine constrained convex optimization.
	\newblock {\em ESAIM Control Optim. Calc. Var.}, 28(0):33, 2022.
	
	\bibitem{Luo_acclerated_2024}
	H.~Luo.
	\newblock Accelerated primal-dual proximal gradient splitting methods for
	convex-concave saddle-point problems.
	\newblock {\em arXiv:2407.20195}, 2024.
	
	\bibitem{luo_universal_2024}
	H.~Luo.
	\newblock A universal accelerated primal\textendash dual method for convex
	optimization problems.
	\newblock {\em J. Optim. Theory Appl.}, 201(1):280--312, 2024.
	
	\bibitem{luo_icpdps_2025}
	H.~Luo.
	\newblock A continuous perspective on the inertial corrected primal-dual
	proximal splitting.
	\newblock {\em Optimization}, pages 1--30, 2025.
	
	\bibitem{luo_differential_2022}
	H.~Luo and L.~Chen.
	\newblock From differential equation solvers to accelerated first-order methods
	for convex optimization.
	\newblock {\em Math. Program.}, 195:735--781, 2022.
	
	\bibitem{luo_unified_2025}
	H.~Luo and Z.~Zhang.
	\newblock A unified differential equation solver approach for separable convex
	optimization: splitting, acceleration and nonergodic rate.
	\newblock {\em Math. Comput.}, 94(356):3009--3041, 2025.
	
	\bibitem{Nesterov1983AMF}
	Y.~Nesterov.
	\newblock A method for solving the convex programming problem with convergence
	rate {$O(1/k^2)$}.
	\newblock {\em Proceedings of the USSR Academy of Sciences}, 269:543--547,
	1983.
	
	\bibitem{Nesterov2004}
	Y.~Nesterov.
	\newblock {\em Introductory Lectures on Convex Optimization: A Basic Course},
	volume~87 of {\em Applied Optimization}.
	\newblock 2004.
	
	\bibitem{2003Smooth}
	Y.~Nesterov.
	\newblock Smooth minimization of non-smooth functions.
	\newblock {\em Math. Program.}, 103(1):127--152, 2005.
	
	\bibitem{Polyak_Some_1964}
	B.~Polyak.
	\newblock Some methods of speeding up the convergence of iteration methods.
	\newblock {\em USSR Comput. Math. Math. Phys.}, 4(5):1--17, 1964.
	
	\bibitem{shi_understanding_2021}
	B.~Shi, S.~S. Du, M.~I. Jordan, and W.~J. Su.
	\newblock Understanding the acceleration phenomenon via high-resolution
	differential equations.
	\newblock {\em Math. Program.}, 195:79--148, 2022.
	
	\bibitem{siegel_accelerated_2019}
	J.~Siegel.
	\newblock Accelerated first-order methods: differential equations and
	{Lyapunov} functions.
	\newblock {\em arXiv:1903.05671}, 2019.
	
	\bibitem{su_dierential_2016}
	W.~Su, S.~Boyd, and E.~J. Cand\`{e}s.
	\newblock A differential equation for modeling {Nesterov}'s accelerated
	gradient method: theory and insights.
	\newblock {\em J. Mach. Learn. Res.}, 17:1--43, 2016.
	
	\bibitem{sun_non-ergodic_2019}
	T.~Sun, P.~Yin, D.~Li, C.~Huang, L.~Guan, and H.~Jiang.
	\newblock Non-ergodic convergence analysis of heavy-ball algorithms.
	\newblock In {\em Proceedings of the {AAAI} {Conference} on {Artificial}
		{Intelligence}}, volume~33, pages 5033--5040, 2019.
	
	\bibitem{valkonen_inertial_2020}
	T.~Valkonen.
	\newblock Inertial, corrected, primal-dual proximal splitting.
	\newblock {\em SIAM J. Optim.}, 30(2):1391--1420, 2020.
	
	\bibitem{wei2024accelerated}
	J.~Wei and L.~Chen.
	\newblock Accelerated over-relaxation heavy-ball method: Achieving global
	accelerated convergence with broad generalization.
	\newblock {\em arXiv:2406.09772}, 2024.
	
	\bibitem{wibisono_variational_2016}
	A.~Wibisono, A.~C. Wilson, and M.~Jordan.
	\newblock A variational perspective on accelerated methods in optimization.
	\newblock {\em Proc. Natl. Acad. Sci. USA}, 113(47):E7351--E7358, 2016.
	
	\bibitem{wilson_lyapunov_2021}
	A.~C. Wilson, B.~Recht, and M.~I. Jordan.
	\newblock A {Lyapunov} analysis of accelerated methods in optimization.
	\newblock {\em J. Mach. Learn. Res.}, 22:1--34, 2021.
	
	\bibitem{Yuan2024AnalyzeAM}
	Y.~Yuan and Y.~Zhang.
	\newblock Analyze accelerated mirror descent via high-resolution {ODE}s.
	\newblock {\em J. Oper. Res. Soc. China}, 2024.
	
\end{thebibliography}

\end{document}